\documentclass[11pt,a4paper]{amsart}
\usepackage{mathrsfs}
\usepackage{tikz-cd}
\usepackage{syntonly}
\usepackage{amsmath}
\usepackage{forest}
\usepackage{xcolor}
\usepackage{amsthm}
\usepackage{amsfonts}
\usepackage{amssymb}
\usepackage{latexsym}
\usepackage{bm}
\usepackage{mathtools}
\usepackage{amscd,amssymb,amsopn,amsmath,amsthm,graphics,amsfonts,mathrsfs,accents,enumerate,verbatim,calc}
\usetikzlibrary{nfold}
\usepackage[colorlinks=true,linkcolor=red,citecolor=blue]{hyperref}
\usepackage[all]{xy}
\usepackage{enumitem}

\date{}
\allowdisplaybreaks[4] \footskip=15pt
\renewcommand{\uppercasenonmath}[1]{}

\renewcommand{\thefootnote}{\fnsymbol{footnote}}

\numberwithin{equation}{section} \theoremstyle{plain}
\newtheorem{lem}{Lemma}[section]
\newtheorem{cor}[lem]{Corollary}
\newtheorem{prop}[lem]{Proposition}
\newtheorem{ex}[lem]{Example}
\newtheorem{thm}[lem]{Theorem}

\newtheorem{construction}[lem]{Construction}

\newtheorem{definition}[lem]{Definition}

\newtheorem{Ctex}[lem]{Counterexample}
\newtheorem{Quest}[lem]{Question}
\newtheorem{Property}[lem]{Property}
\newtheorem{Properties}[lem]{Properties}
\newtheorem{Subprops}{}[lem]
\newtheorem{Para}[lem]{}

\newtheorem{remark}[lem]{Remark}
\newtheorem{rem}[lem]{Remark}

\newtheorem*{ThA}{Theorem A}
\newtheorem*{ThB}{Theorem B}
\newtheorem*{ThC}{Theorem C}

\newtheorem*{ack*}{ACKNOWLEDGEMENTS}

\newcommand{\HomC}{\mathrm{Hom}_{\mC}}
\newcommand{\HomA}{\mathrm{Hom}_{\mA}}
\newcommand{\Tree}{\mathfrak T}
\newcommand{\modu}{\operatorname{mod}}
\newcommand{\add}{\operatorname{add}}
\newcommand{\mapsfrom}{\mathrel{\reflectbox{\ensuremath{\mapsto}}}}

\newcommand{\pf}{\noindent\begin {proof}}
\newcommand{\epf}{\end{proof}}

\newcommand{\X}{\mathcal{X}}

\newcommand{\mC}{\mathcal{C}}

\newcommand{\mA}{\mathcal{A}}
\newcommand{\B}{\mathcal{B}}
\newcommand{\mD}{\mathcal{D}}
\newcommand{\mX}{\mathcal{X}}
\newcommand{\mF}{\mathcal{F}}

\RequirePackage{tikz-cd}
\RequirePackage{amssymb}
\usetikzlibrary{calc,decorations.pathmorphing,spath3}

\tikzset{
	curve/.style={
		settings={#1},
		to path={
			(\tikztostart)
			.. controls
			($(\tikztostart)!\pv{pos}!(\tikztotarget)!\pv{height}!270:(\tikztotarget)$)
			and
			($(\tikztostart)!1-\pv{pos}!(\tikztotarget)!\pv{height}!270:(\tikztotarget)$)
			.. (\tikztotarget)\tikztonodes
		}
	},
	settings/.code={
		\tikzset{quiver/.cd,#1}
		\def\pv##1{\pgfkeysvalueof{/tikz/quiver/##1}}
	},
	quiver/.cd,
	pos/.initial=0.35,
	height/.initial=0
}

\tikzset{
	between/.style n args={2}{
		/tikz/execute at end to={
			\tikzset{
				spath/split at keep middle={current}{#1}{#2}
			}
		}
	}
}

\tikzset{
	tail reversed/.code={
		\pgfsetarrowsstart{tikzcd to}
	},
	2tail/.code={
		\pgfsetarrowsstart{Implies[reversed]}
	},
	2tail reversed/.code={
		\pgfsetarrowsstart{Implies}
	},
	no body/.style={
		/tikz/dash pattern=on 0 off 1mm
	}
}

\begin{document}
\begin{center}
{\Large  \bf Homology and higher-order intermediate  categories for proper abelian subcategories}
\end{center}
\medskip
\centerline{Zheng-Gang Wu \qquad Xianhui Fu\textsuperscript{*}}

\begingroup
\renewcommand{\thefootnote}{\fnsymbol{footnote}}
\footnotetext[1]{Corresponding author.}
\endgroup

\bigskip
\medskip
\centerline { \bf  Abstract}
\leftskip10truemm \rightskip10truemm \noindent
\noindent
Without needing a $t$-structure, we establish natural long exact homology sequences  on finite extension windows with $\ell+1$ layers
of extension-closed proper abelian subcategories,
assuming vanishing of negative self-extensions
in degrees $1,\ldots,\ell+1$,
and prove uniform sharpness of this bound. In the Krull-Schmidt setting, we classify intermediate subcategories
stable under taking pure homology layers by compatible sequences. 

\medskip

\noindent\textbf{2020 Mathematics Subject Classification.}
Primary 18G80; Secondary 18E10, 18E40.

\noindent
\textbf{Keywords:} proper abelian subcategory; long exact homology sequence; \(s\)-torsion pair; torsion-free class;
 intermediate subcategory; negative cluster category.

\leftskip0truemm
\rightskip0truemm
\setcounter{tocdepth}{1}
\tableofcontents

\section {Introduction}
%\bigskip
Proper abelian subcategories retain the compatibility between short exact sequences and distinguished triangles which is a characteristic of hearts of $t$-structures, but need not themselves be hearts; see
Jørgensen \cite{Jorgensen2022,Jorgensen2021}. This raises the question of how to construct homology and classify intermediate
subcategories without a $t$-structure.

Let $\mA$ be an extension-closed proper abelian subcategory in a triangulated category $\mC$. Set the condition $E_r$, which means $\mathrm{Hom}(\Sigma^i\mA, \mA)=0, 1\leq i\leq r$. Korteg{\aa}rd \cite{Kortegaard2025} established a six-term sequence on
$\Sigma\mathcal A*\mathcal A$ under Hom-vanishing condition $E_2$ and used it to classify
intermediate subcategories under an exchange hypothesis. For arbitrary finite extension windows $\mD_{\ell}:=\Sigma^\ell\mA*\cdots*\Sigma\mA*\mA$, his thesis
\cite[Project D, Theorem~4.5]{Kortegard2024} establishes a finite-window long exact
sequence under $E_{\ell+3}$. 

Our first main result proves the same long exact sequence under the strictly weaker condition $(E_{\ell+1})$, using layer homological functors $H_i$ constructed in Construction \ref{cons:canonical-tower}; thus the vanishings in degrees $(\ell+2)$ and $(\ell+3)$ are not required.

\begin{ThA}[\textit{=Theorem \ref{thm:main-long-exact}}]
	Let
	\begin{equation}
		\Delta=(X\xrightarrow fY\xrightarrow gZ\xrightarrow h\Sigma X)
	\end{equation}
	be a triangle with $X,Y,Z\in\mD_\ell$.  There are canonical connecting morphisms 
	\[
	\delta_i^\Delta:H_i(Z)\longrightarrow H_{i-1}(X)
	\qquad(1\leq i\leq\ell),
	\]
	for which 
	\begin{align*}
		0\longrightarrow H_\ell(X)&\longrightarrow H_\ell(Y)
		\longrightarrow H_\ell(Z)\xrightarrow{\delta_\ell^\Delta}
		H_{\ell-1}(X)\longrightarrow\cdots\\
		&\cdots\longrightarrow H_0(X)\longrightarrow H_0(Y)
		\longrightarrow H_0(Z)\longrightarrow0
	\end{align*}
	is exact in $\mA$.  Moreover, the sequence is natural with respect to
	morphisms of triangles.
\end{ThA}

For $\ell=1$, this recovers Korteg{\aa}rd's original six-term sequence \cite[Lemma 3.3]{Kortegaard2025}.
Proposition~\ref{prop:sharp} shows that the bound $(E_{\ell+1})$ is
uniformly sharp: the weaker condition $(E_\ell)$ is not sufficient
in general.

We continue to assume $(E_{\ell+1})$ and now suppose that
$\mathcal C$ is Krull--Schmidt. 

If the highest-layer exchange equality holds, we give the classification of $\ell$-layer intermediate subcategories using torsion-free classes in $\mA$. This generalizes the classification of intermediate subcategories in \cite[Corollary 4.4]{Kortegaard2025}.

\begin{ThB}[\textit{=Theorem \ref{thm 6.3}}]
		Assume the highest-layer exchange equality
	$$\Sigma^\ell\mA*\mD_{\ell-1}=\mD_{\ell-1}*\Sigma^\ell\mA.$$
	Then there is a bijection
\[
\begin{array}{r@{\;}c@{\;}l}
	\left\{
	\begin{array}{c}
		\ell\text{-layer intermediate}\\
		\text{subcategories } \mathcal X
	\end{array}
	\right\}
	&
	\overset{1:1}{\longleftrightarrow}
	&
	\left\{
	\begin{array}{c}
		\text{torsion-free classes}\\
		\mathcal F\subseteq \mathcal A
	\end{array}
	\right\}
	\\[1.6em]
	\mathcal X
	&
	\mapsto
	&
	H_\ell(\mathcal X)
	\\[0.6em]
	\Sigma^\ell\mathcal F * \mathcal D_{\ell-1}
	&
	\mapsfrom
	&
	\mathcal F
\end{array}
\]
Moreover, this bijection is an isomorphism of posets with respect to inclusion.
\end{ThB}

We
classify subcategories $\mA\subseteq \mX \subseteq \mD_\ell$ closed under extensions, direct summands,
and taking pure homology layers, using compatible sequences of subcategories of $\mA$ in the sense of Definition \ref{compatible}.

\begin{ThC}[\textit{=Theorem \ref{thm:layers}}]
	There is an isomorphism of posets between compatible sequences, ordered
	componentwise by inclusion, and subcategories
	$\mA\subseteq\X\subseteq \mD_\ell$ closed under extensions, direct summands,
	and taking pure homology layers. The maps are
	\begin{align*}
		(\B_0,\ldots,\B_\ell)&\longmapsto
		\X_{\B}:=\{x\in \mD_\ell\mid H_i(x)\in\B_i
		\text{ for all }i\},\\
		\X&\longmapsto
		\bigl(\{a\in\mA\mid\Sigma^ia\in\X\}\bigr)_{i=0}^{\ell}.
	\end{align*}
	Moreover,
	\[
	\X_{\B}=\Sigma^\ell\B_\ell*\cdots*\Sigma\B_1*\B_0.
	\]
\end{ThC}

Sections~2--4 develop the homology construction and prove the exactness.
Sections~5--7 establish the classification results. Section~8
treats hearts, negative cluster categories, and non-semisimple
periodic derived orbit examples with nonzero connecting morphisms.

\section {Preliminaries}
%\bigskip

Throughout, all subcategories are assumed to be full, replete, and
additive. Let $\mathcal C$ be a triangulated category with suspension
functor $\Sigma$. For subcategories $\mathcal X,\mathcal Y\subseteq
\mathcal C$, write
\[
\mathcal X*\mathcal Y
=
\left\{
z\in\mathcal C\ \middle|\
\text{there is a triangle }
x\longrightarrow z\longrightarrow y\longrightarrow\Sigma x,
\ x\in\mathcal X,\ y\in\mathcal Y
\right\}.
\]
The operation $*$ is associative. A subcategory $\mathcal Z$ is
extension closed if $\mathcal Z*\mathcal Z\subseteq\mathcal Z$.

According to \cite[Example~2.4]{Adachi2023} and \cite{NP}, every extension-closed subcategory $\mathcal Z\subseteq\mathcal C$
inherits the restricted extriangulated structure with the negative first
extension
\[
\mathbb E^{-1}(x,y):=\operatorname{Hom}_{\mathcal C}
(x,\Sigma^{-1}y),
\qquad x,y\in\mathcal Z.
\]
We use only the following specialization of the theory of
$s$-torsion pairs from~\cite{Adachi2023}.

\begin{definition}[{\cite[Definition~3.1]{Adachi2023}}]
	Let $\mathcal Z\subseteq\mathcal C$ be extension closed. A pair
	$(\mathcal T,\mathcal F)$ of subcategories of $\mathcal Z$ is an
	$s$-torsion pair in $\mathcal Z$ if
	\[
	\mathcal Z=\mathcal T*\mathcal F,\qquad
	\operatorname{Hom}_{\mathcal C}(\mathcal T,\mathcal F)=0,\qquad
	\operatorname{Hom}_{\mathcal C}
	(\mathcal T,\Sigma^{-1}\mathcal F)=0.
	\]
\end{definition}

\begin{prop}[{\cite[Proposition~3.7]{Adachi2023}}]
	Let $(\mathcal T,\mathcal F)$ be an $s$-torsion pair in
	$\mathcal Z$. For every $z\in\mathcal Z$, there is a triangle
	\[
	Tz\xrightarrow{\iota_z}z\xrightarrow{\pi_z}Fz
	\longrightarrow\Sigma Tz,
	\qquad Tz\in\mathcal T,\quad Fz\in\mathcal F,
	\]
	which is unique up to isomorphism. Moreover, the assignments
	$z\mapsto Tz$ and $z\mapsto Fz$ define additive functors
	\[
	T\colon\mathcal Z\longrightarrow\mathcal T,
	\qquad
	F\colon\mathcal Z\longrightarrow\mathcal F,
	\]
	where $T$ is right adjoint to the inclusion
	$\mathcal T\hookrightarrow\mathcal Z$ and $F$ is left adjoint to the
	inclusion $\mathcal F\hookrightarrow\mathcal Z$.
\end{prop}

The uniqueness and functoriality in this proposition are the only
general properties of $s$-torsion pairs needed below.

\begin{definition}[{\cite[Definition 2.2 (ii)]{Jorgensen2022}}]
		Let $\mA\subseteq\mC$ be an additive subcategory. $\mA$ is a {\it proper abelian subcategory} if it is abelian in such a way that $a \overset{\alpha}{\hookrightarrow} a' \overset{\alpha'}{\twoheadrightarrow} a''$ is a short exact sequence in $\mA$ if and only if there is a triangle $a\xrightarrow{\alpha} a'\xrightarrow{\alpha'}a''\to \Sigma a$ in $\mC$.
\end{definition}

\begin{ex}[{\cite[Example 5.2] {Kortegaard2025}}, \cite{Beilinson2018}]
		The heart of a $t$-structure in a triangulated category is a proper abelian subcategory. But the
		converse need not hold.
\end{ex}

We state the following result in homological notation.

\begin{prop}[{\cite[Lemma 3.3] {Kortegaard2025}}]\label{prop 2.19} 
	Let $\mA$ be an extension-closed proper abelian subcategory in $\mC$. Assume that $\mC(\Sigma^{i}\mA, \mA)=0$, for $i=1, 2$. Then there exist two functors $H_1, H_0: \Sigma\mA*\mA\to \mA$ satisfying the following. Suppose 
	$$\xymatrix{
		c \ar[r]^{f} & c'\ar[r]^{g} & c'' \ar[r] & \Sigma c} $$
		is a distinguished triangle with $c, c', c''\in \Sigma\mA*\mA$,  there exists a morphism $\delta: H_1(c'')\to H_0(c)$, such that 
		\[\begin{tikzcd}
			0 & {H_1(c)} & {H_1(c')} & {H_1(c'')} & {H_0(c)} & {H_0(c')} & {H_0(c'')} & 0
			\arrow[from=1-1, to=1-2]
			\arrow["{H_1(f)}", from=1-2, to=1-3]
			\arrow["{H_1(g)}", from=1-3, to=1-4]
			\arrow["\delta", from=1-4, to=1-5]
			\arrow["{H_0(f)}", from=1-5, to=1-6]
			\arrow["{H_0(g)}", from=1-6, to=1-7]
			\arrow[from=1-7, to=1-8]
		\end{tikzcd}\]
		is an exact sequence in $\mA$.
\end{prop}
\vskip 15pt

\section{Homology for proper abelian subcategories}\label{section 3}

Throughout this section, let $\ell\geq 1$,  $\mC$ be a triangulated category with suspension functor $\Sigma$, and let $\mA\subseteq\mC$ be a full, replete, extension-closed proper abelian subcategory.  We impose the single standing hypothesis
\begin{equation}\tag{$E_{\ell+1}$}\label{eq:global-vanishing}
	\mathrm{Hom}_{\mC}(\Sigma^p\mA,\mA)=0
	\qquad\text{for }1\leq p\leq\ell+1.
\end{equation}
For $a\leq b$, put
$$\mD_{[b, a]}:=\Sigma^b\mA*\Sigma^{b-1}\mA*\cdots*\Sigma^a\mA,$$
and we abbreviate $\mD_i:=\mD_{[i, 0]}$ and $\mA_i:=\Sigma^i\mA$.

\subsection{Exchange, extension closure, and orthogonality}
\begin{lem}\label{lem 5.1}
	For $0\leq p< q\leq \ell$, one has 
	$\mA_p*\mA_q\subseteq \mA_q *\mA_p.$
\end{lem}
\begin{proof}
	For any $x\in \mA_p*\mA_q$, choose a distinguished triangle
	\[\begin{tikzcd}
		{\Sigma^p a_1} & x & {\Sigma^qa_2} & {\Sigma^{p+1} a_1}
		\arrow[from=1-1, to=1-2]
		\arrow[from=1-2, to=1-3]
		\arrow["\omega", from=1-3, to=1-4]
	\end{tikzcd}\]
	where $a_1, a_2\in\mA$.
	
	If $q\geq p+2$, then $\omega\in \mathrm{Hom}_{\mC}(\mA_q, \mA_{p+1})\cong\mathrm{Hom}_{\mC}(\mA_{q-p-1}, \mA)=0$. Hence the above triangle is split, which shows that $x\in \mA_q*\mA_p$.
	
	If $q=p+1$, after rotating and shifting by $\Sigma^{-p}$, we have the following distinguished triangle
	\[\begin{tikzcd}
		{a_2} & {a_1} & {\Sigma^{-p}x} & {\Sigma a_2}.
		\arrow["f", from=1-1, to=1-2]
		\arrow[from=1-2, to=1-3]
		\arrow[from=1-3, to=1-4]
	\end{tikzcd}\]
	Properness of \(\mathcal A\) and the octahedral axiom give the following diagram.
	\[\begin{tikzcd}
		{a_2} & {\mathrm{Im}f} & {\Sigma\mathrm{Ker}f} & {} \\
		{a_2} & {a_1} & {\Sigma^{-p}x} & {} \\
		& {\mathrm{Coker}f} & {\mathrm{Coker}f} \\
		& {} & {}
		\arrow[from=1-1, to=1-2]
		\arrow[equals, from=1-1, to=2-1]
		\arrow[from=1-2, to=1-3]
		\arrow[from=1-2, to=2-2]
		\arrow["{+1}", from=1-3, to=1-4]
		\arrow[dashed, from=1-3, to=2-3]
		\arrow["f", from=2-1, to=2-2]
		\arrow[from=2-2, to=2-3]
		\arrow[from=2-2, to=3-2]
		\arrow["{+1}", from=2-3, to=2-4]
		\arrow[dashed, from=2-3, to=3-3]
		\arrow[equals, from=3-2, to=3-3]
		\arrow["{+1}", from=3-2, to=4-2]
		\arrow["{+1}", from=3-3, to=4-3]
	\end{tikzcd}\]
	The resulting third octahedral triangle shows that $x\in \mA_q*\mA_p$.
\end{proof}

\begin{lem}\label{lem 3.2}
	Let $\mathcal X_1,\ldots,\mathcal X_s$ and
	$\mathcal Y_1,\ldots,\mathcal Y_t$ be subcategories of $\mC$.  If
	\[
	\mathrm{Hom}_{\mC}(\mathcal X_i,\mathcal Y_j)=0
	\qquad(1\leq i\leq s,\ 1\leq j\leq t),
	\]
	then
	\[
	\mathrm{Hom}_{\mC}(\mathcal X_1*\cdots*\mathcal X_s,
	\mathcal Y_1*\cdots*\mathcal Y_t)=0.
	\]
\end{lem}

\begin{proof}
	Induct on $s+t$.  Applying $\mathrm{Hom}_{\mC}(-,y)$ to a triangle proves closure of
	the left orthogonal $\prescript{\perp}{}{y}$ under extensions.  Applying
	$\mathrm{Hom}_{\mC}(x,-)$ proves the dual assertion for the right variable.  Iterating
	these two observations gives the claim.
\end{proof}

\begin{prop}\label{prop 3.3}
	For $b-a\leq \ell$, the subcategory $\mD_{[b,a]}$ is extension closed.
	Consequently it inherits the restricted extriangulated structure and the
	negative first extension
	\[
	\mathbb E^{-1}(x,y):=\mathrm{Hom}_{\mC}(x,\Sigma^{-1}y).
	\]
\end{prop}

\begin{proof}
	Each $\Sigma^j\mA$ is extension closed.  Starting with
	\[
	(\Sigma^b\mA*\cdots*\Sigma^a\mA)
	*(\Sigma^b\mA*\cdots*\Sigma^a\mA),
	\]
	use associativity of $*$ and repeatedly apply Lemma
	\ref{lem 5.1} to move the second copy of $\Sigma^b\mA$ to
	the left.  Combine the adjacent copies, and repeat with the remaining
	layers.  This proves
	$\mD_{[b,a]}*\mD_{[b,a]}\subseteq\mD_{[b,a]}$.
	The final assertion is the standard restriction of the triangulated
	extriangulated structure to an extension-closed subcategory.
\end{proof}

\subsection{Block decompositions and homology functors}

\begin{prop}\label{prop 3.4}
	Let $a\leq b$ with $b-a\leq \ell$.  For every
	$a\leq i\leq b+1$, the pair
	\begin{equation}
		\bigl(\mD_{[b,i]},\mD_{[i-1,a]}\bigr)
	\end{equation}
	is an $s$-torsion pair in $\mD_{[b,a]}$.  The cases $i=a$ and
	$i=b+1$ are the two trivial endpoint pairs.  In particular,
	\begin{equation}
		(\Sigma^i\mA,\mD_{i-1})
	\end{equation}
	is an $s$-torsion pair in $\mD_i$ for $0\leq i\leq \ell$.
\end{prop}

\begin{proof}
	Associativity gives
	$\mD_{[b,a]}=\mD_{[b,i]}*\mD_{[i-1,a]}$; this also proves the endpoint
	cases.  Otherwise, if $i\leq p\leq b$ and $a\leq q\leq i-1$, then
	\[
	\mathrm{Hom}_{\mC}(\Sigma^p\mA,\Sigma^q\mA)
	=\mathrm{Hom}_{\mC}(\Sigma^{p-q}\mA,\mA)=0
	\]
	and
	\[
	\mathrm{Hom}_{\mC}(\Sigma^p\mA,\Sigma^{-1}\Sigma^q\mA)
	=\mathrm{Hom}_{\mC}(\Sigma^{p-q+1}\mA,\mA)=0
	\]
	hold by
	\eqref{eq:global-vanishing}, and Lemma
	\ref{lem 3.2} proves the two orthogonality axioms.
\end{proof}

Fix $0\leq n\leq \ell$.  For $0\leq i\leq n+1$, apply
Proposition \ref{prop 3.4} to the block pair
\[
(\mD_{[n,i]},\mD_{i-1})
\]
inside $\mD_n$.  The uniqueness theorem for $s$-torsion decompositions
gives functors
\[
U_i^{[n]}:\mD_n\longrightarrow\mD_{[n,i]},
\qquad
L_{i-1}^{[n]}:\mD_n\longrightarrow\mD_{i-1}
\]
and natural triangles
\begin{equation}\label{eq:block-decomposition}
	U_i^{[n]}(x)\xrightarrow{e_i^x}x
	\xrightarrow{p_{i-1}^x} L_{i-1}^{[n]}(x)
	\longrightarrow\Sigma U_i^{[n]}(x).
\end{equation}
Use the endpoint conventions
\[
U_0^{[n]}=\mathrm{Id}_{\mD_n},\quad L_{-1}^{[n]}=0,
\qquad U_{n+1}^{[n]}=0,\quad L_n^{[n]}=\mathrm{Id}_{\mD_n}.
\]
We normalize the two trivial endpoint triangles by
\[
e_0^x=1_x,\quad p_{-1}^x=0,
\qquad e_{n+1}^x=0,\quad p_n^x=1_x.
\]

\begin{construction}\label{cons:canonical-tower}
	For $0\leq i\leq n$, apply the elementary $s$-torsion pair
	$(\Sigma^i\mA,\mD_{i-1})$ in $\mD_i$ to the functorial residual
	$L_i^{[n]}(x)$.  Here the case $i=0$ means the trivial pair
	$(\mA,0)$, and $L_n^{[n]}=\mathrm{Id}$.  Since the composite of the two left
	adjoints onto $\mD_i$ and then $\mD_{i-1}$ is the left adjoint onto
	$\mD_{i-1}$, uniqueness identifies its residual term with
	$L_{i-1}^{[n]}(x)$.  Define the additive functor
	$H_i^{[n]}:\mD_n\to\mA$ by the resulting natural triangle
	\begin{equation}\label{eq:lower-layer-tower}
		\Sigma^iH_i^{[n]}(x)\xrightarrow{\alpha_i^x}L_i^{[n]}(x)
		\xrightarrow{\lambda_i^x}L_{i-1}^{[n]}(x)
		\longrightarrow\Sigma^{i+1}H_i^{[n]}(x).
	\end{equation}
	Thus the functors $H_i^{[n]}$ and the entire residual tower are defined
	only from the functorial $s$-torsion decompositions and their uniqueness.
	For $n=\ell$, write simply $H_i$.
\end{construction}

\begin{prop}
	For $0\leq i\leq n$, there is a unique natural map
	\[
	d_i^x:U_{i+1}^{[n]}(x)\longrightarrow U_i^{[n]}(x),
	\qquad e_i^xd_i^x=e_{i+1}^x,
	\]
	and it fits into a natural triangle
	\begin{equation}\label{eq:upper-layer-tower}
		U_{i+1}^{[n]}(x)\xrightarrow{d_i^x}U_i^{[n]}(x)
		\longrightarrow\Sigma^iH_i^{[n]}(x)
		\longrightarrow\Sigma U_{i+1}^{[n]}(x).
	\end{equation}
	Write $\beta_i^x:U_i^{[n]}(x)\to\Sigma^iH_i^{[n]}(x)$ for the
	second arrow of \eqref{eq:upper-layer-tower}, and write
	$p_i^x:x\to L_i^{[n]}(x)$ for the second arrow of the block
	decomposition at the cut $i+1$.  These maps form a natural commutative
	square
	\begin{equation}\label{eq:nested-block-square}
		\begin{tikzcd}[column sep=large]
			U_i^{[n]}(x) \arrow[r,"e_i^x"] \arrow[d,"\beta_i^x"'] &
			x \arrow[d,"p_i^x"]\\
			\Sigma^iH_i^{[n]}(x) \arrow[r,"\alpha_i^x"'] &
			L_i^{[n]}(x).
		\end{tikzcd}
	\end{equation}
\end{prop}

\begin{proof}
	For $i<n$, apply the block $s$-torsion pair
	$(\mD_{[n,i+1]},\Sigma^i\mA)$ inside $\mD_{[n,i]}$; for $i=n$, use the
	trivial pair $(0,\Sigma^n\mA)$.  Compose its right adjoint with
	$U_i^{[n]}:\mD_n\to\mD_{[n,i]}$.  The composite is right adjoint to the
	inclusion $\mD_{[n,i+1]}\hookrightarrow\mD_n$, so it is canonically
	$U_{i+1}^{[n]}$.  Its counit is exactly the unique map $d_i^x$ satisfying
	$e_i^xd_i^x=e_{i+1}^x$.  Thus its residual is a functorial object of
	$\Sigma^i\mA$ and gives a natural triangle of the form
	\eqref{eq:upper-layer-tower}, initially with some functor
	$G_i^{[n]}:\mD_n\to\mA$ in place of $H_i^{[n]}$.
	
	Let $\gamma_i^x:U_i^{[n]}(x)\to\Sigma^iG_i^{[n]}(x)$ be the pure
	quotient map.  Apply \textup{(TR3)} to the square formed by
	$1_{U_{i+1}^{[n]}(x)}$ and $e_i^x$ in the upper layer triangle and the
	block decomposition at the cut $i+1$.  This gives
	\[
	c_i^x:\Sigma^iG_i^{[n]}(x)\longrightarrow L_i^{[n]}(x),
	\qquad c_i^x\gamma_i^x=p_i^xe_i^x.
	\]
	It is unique, because its ambiguity lies in
	\begin{equation}\label{eq:nested-comparison-vanishing}
		\mathrm{Hom}_{\mC}\bigl(\Sigma U_{i+1}^{[n]}(x),L_i^{[n]}(y)\bigr)=0
		\qquad(x,y\in\mD_n).
	\end{equation}
	On pure layers the positive exponents range from $2$ to at most
	$n+1\leq\ell+1$, so this follows from
	\eqref{eq:global-vanishing} and Lemma
	\ref{lem 3.2}.  Applying the same uniqueness to a
	morphism $x\to y$ shows that the maps $c_i^x$ are natural.
	
	There is also a unique natural map
	$\lambda_i^x:L_i^{[n]}(x)\to L_{i-1}^{[n]}(x)$ satisfying
	$\lambda_i^xp_i^x=p_{i-1}^x$: it is the factorization through the unit
	$p_i^x$ of the left adjoint onto $\mD_i$.  Apply the octahedral axiom
	objectwise to $U_{i+1}^{[n]}(x)\to U_i^{[n]}(x)\to x$.  It shows that
	\[
	\Sigma^iG_i^{[n]}(x)\xrightarrow{c_i^x}L_i^{[n]}(x)
	\xrightarrow{\lambda_i^x}L_{i-1}^{[n]}(x)\longrightarrow
	\Sigma^{i+1}G_i^{[n]}(x)
	\]
	is a triangle; the first two arrows are the canonical ones above by their
	uniqueness.  Thus it is an $(\Sigma^i\mA,\mD_{i-1})$-decomposition of
	$L_i^{[n]}(x)$.  Its unique comparison with the functorial elementary
	decomposition \eqref{eq:lower-layer-tower} identifies $G_i^{[n]}$ with
	$H_i^{[n]}$ naturally.  Transporting the upper triangle along this
	isomorphism gives \eqref{eq:upper-layer-tower}, and
	$c_i^x\gamma_i^x=p_i^xe_i^x$ becomes
	\eqref{eq:nested-block-square}.  The construction is independent of the
	chosen objectwise octahedron and does not choose octahedra functorially.
\end{proof}

At the endpoint $i=n$, \eqref{eq:upper-layer-tower} is
$0\to U_n^{[n]}(x)\xrightarrow{\beta_n^x}
\Sigma^nH_n^{[n]}(x)\to0$.  From now on we use this canonical
isomorphism as an identification
\begin{equation}\label{eq:top-endpoint-normalization}
	U_n^{[n]}(x)=\Sigma^nH_n^{[n]}(x),
	\qquad \beta_n^x=1,
	\qquad e_n^x=\alpha_n^x.
\end{equation}
This fixes the endpoint normalization used in the definition of the top
connecting morphism.

\begin{prop}\label{prop 3.7}
	Let $0\leq n\leq\ell$.
	\begin{enumerate}[label=\textup{(\arabic*)}]
		\item For $a\in\mA$ and $0\leq i,j\leq n$, naturally in $a$,
		\[
		H_i^{[n]}(\Sigma^ja)\cong
		\begin{cases}a,&i=j,\\0,&i\neq j.\end{cases}
		\]
		\item If $x\in\mD_m\subseteq\mD_n$, then
		\[
		U_i^{[n]}(x)\cong U_i^{[m]}(x),\qquad
		L_{i-1}^{[n]}(x)\cong L_{i-1}^{[m]}(x),\qquad
		H_i^{[n]}(x)\cong H_i^{[m]}(x)\quad(0\leq i\leq m),
		\]
		naturally, while for $m<i\leq n$ one has
		$U_i^{[n]}(x)=0$, $L_{i-1}^{[n]}(x)\cong x$, and
		$H_i^{[n]}(x)=0$.  These comparisons identify the counits, units, the
		maps $d_i$, and both layer triangles.  Thus adding a zero top layer
		does not change any pre-existing decomposition or layer map.  Suspending
		these canonical comparisons gives the identical statement for every
		shifted interval.
		\item Put $x^{<n}:=L_{n-1}^{[n]}(x)$.  Then
		\begin{equation}\label{eq:top-layer-decomposition}
			\Sigma^nH_n^{[n]}(x)\longrightarrow x\xrightarrow{\pi_x^n}x^{<n}
			\longrightarrow\Sigma^{n+1}H_n^{[n]}(x)
		\end{equation}
		is natural, and
		\[
		H_i^{[n]}(x)\xrightarrow[\sim]{H_i(\pi_x^n)}
		H_i^{[n-1]}(x^{<n})
		\qquad(0\leq i<n).
		\]
		\item For $0\leq m<n$,
		\[
		x\in\mD_m
		\quad\Longleftrightarrow\quad
		H_{m+1}^{[n]}(x)=\cdots=H_n^{[n]}(x)=0.
		\]
		In particular, $\mathrm{Ker} H_n^{[n]}=\mD_{n-1}$.
	\end{enumerate}
\end{prop}

\begin{proof}
	For a pure object, each elementary $s$-torsion decomposition is either
	$u\xrightarrow{1}u\to0$ or $0\to u\xrightarrow{1}u$ until its unique
	degree is reached; this proves (1).
	
	If $x\in\mD_m$, the unit $x\to L_m^{[n]}(x)$ is an isomorphism: it is the
	unit of the left adjoint to the inclusion $\mD_m\hookrightarrow\mD_n$
	evaluated on an object already in $\mD_m$.  The residual tower below level
	$m$ is therefore the same tower as the one constructed inside $\mD_m$,
	while every elementary torsion part above level $m$ is zero.  Uniqueness
	of the units and of the elementary decompositions makes all these
	identifications natural.  Likewise, the block decomposition of $x$ at a
	cut $i\leq m$ inside $\mD_m$ is already a decomposition for the pair
	$(\mD_{[n,i]},\mD_{i-1})$ inside $\mD_n$; at a cut above $m$ it is the
	trivial triangle $0\to x\xrightarrow{1}x$.  Uniqueness of the block
	decompositions identifies $U_i,L_{i-1},e_i,p_{i-1}$ and therefore also
	$d_i,\alpha_i,\beta_i$.  This proves every assertion in (2).
	
	The case $i=n$ of \eqref{eq:lower-layer-tower} is
	\eqref{eq:top-layer-decomposition}.  Moreover, for $i<n$, the composite
	of the left adjoint $L_{n-1}^{[n]}:\mD_n\to\mD_{n-1}$ with the residual
	left adjoint inside $\mD_{n-1}$ is the residual left adjoint inside
	$\mD_n$.  Uniqueness of left adjoints therefore identifies the two lower
	towers, and the comparison is induced by the unit
	$\pi_x^n:x\to x^{<n}$.  Applying their elementary torsion functors proves
	(3), including the displayed assertion about $H_i(\pi_x^n)$.  Finally,
	repeatedly apply (3): a zero top layer makes the top decomposition triangle
	identify an object with its lower residual.  This proves (4).
\end{proof}

We suppress $[n]$ when the ambient category $\mD_n$ is clear.
The functors $H_i^{[n]}$, their block decomposition terms
$U_i^{[n]}$, and their uniqueness are the only homological tools
used in Section~\ref{section 4}.

We next show that recursive block decompositions recover
the same layer functors, independently of the order of cuts.

\subsection{Catalan coherence and path independence}

We use the layer homology functors already constructed from the
canonical layer tower,
\[
H_j^{[b,a]}:\mD_{[b,a]}\longrightarrow\mA
\qquad(a\leq j\leq b).
\]
The nested-block layer triangles characterize them from either side:
\begin{align}
	U_{j+1}^{[b,a]}x&\longrightarrow U_j^{[b,a]}x
	\longrightarrow\Sigma^jH_j^{[b,a]}(x)
	\longrightarrow\Sigma U_{j+1}^{[b,a]}x,
	\label{eq:upper-layer-triangle}\\
	\Sigma^jH_j^{[b,a]}(x)&\longrightarrow L_j^{[b,a]}x
	\longrightarrow L_{j-1}^{[b,a]}x
	\longrightarrow\Sigma^{j+1}H_j^{[b,a]}(x).
	\label{eq:lower-layer-triangle}
\end{align}

\begin{lem}
	\label{lem:block-restriction}
	Let $a<r\leq b$.
	\begin{enumerate}
		\item For $r\leq j\leq b$, there is a canonical natural
		isomorphism
		\[
		H_j^{[b,r]}U_r^{[b,a]}
		\xrightarrow{\ \sim\ }
		H_j^{[b,a]}.
		\tag{U}
		\]
		\item For $a\leq j<r$, there is a canonical natural isomorphism
		\[
		H_j^{[r-1,a]}L_{r-1}^{[b,a]}
		\xrightarrow{\ \sim\ }
		H_j^{[b,a]}.
		\tag{L}
		\]
	\end{enumerate}
	These isomorphisms identify the corresponding layer triangles and are
	compatible with further restriction to nested subintervals.
\end{lem}

\begin{proof}
	Suppose first that $r\leq j\leq b$.  The composite
	\[
	\mD_{[b,j]}\hookrightarrow\mD_{[b,r]}
	\hookrightarrow\mD_{[b,a]}
	\]
	has right adjoint
	\[
	U_j^{[b,r]}U_r^{[b,a]}.
	\]
	On the other hand, its right adjoint obtained directly from the block
	decomposition at the cut $j$ is $U_j^{[b,a]}$.  Uniqueness of right
	adjoints, with their counits, gives a canonical natural isomorphism
	\begin{equation}
		U_j^{[b,r]}U_r^{[b,a]}
		\xrightarrow{\ \sim\ }U_j^{[b,a]}.
		\label{eq:upper-adjoint-comparison}
	\end{equation}
	The same statement holds with $j+1$ in place of $j$.  These two
	comparisons identify the first morphisms in the corresponding upper
	layer triangles \eqref{eq:upper-layer-triangle}.  The uniqueness in
	the nested-block layer construction then identifies their pure terms,
	giving (U).  This comparison is natural because all the adjunction
	comparisons and layer triangles are natural.
	\[\begin{tikzcd}
		{U_{j+1}^{[b,a]}x} & {U_{j}^{[b,a]}x} & {\Sigma^jH_j^{[b,a]}(x)} & {} \\
		{U_{j+1}^{[b,r]}U_r^{[b,a]}x} & {U_{j}^{[b,r]}U_r^{[b,a]}x} & {\Sigma^jH_j^{[b,r]}U_r^{[b,a]}(x)} & {}
		\arrow[from=1-1, to=1-2]
		\arrow["\cong"', from=1-1, to=2-1]
		\arrow[from=1-2, to=1-3]
		\arrow["\cong"', from=1-2, to=2-2]
		\arrow["{+1}", from=1-3, to=1-4]
		\arrow["\cong"', dashed, from=1-3, to=2-3]
		\arrow[from=2-1, to=2-2]
		\arrow[from=2-2, to=2-3]
		\arrow["{+1}", from=2-3, to=2-4]
	\end{tikzcd}\]
	The proof of remaining canonical natural isomorphism can be similarly obtained.
\end{proof}

We now formalize the possible recursive extraction paths.  An
\emph{interval decomposition tree} for $[b,a]$ is a planar full
binary tree whose leaves, from left to right, are labelled
\[
b,b-1,\ldots,a,
\]
and whose vertices correspond to consecutive subintervals.  More
precisely, the one-leaf interval $[a,a]$ has one tree.  If $a<b$, a
tree is obtained by choosing a root cut $a<r\leq b$, a tree for
$[b,r]$, and a tree for $[r-1,a]$.  Hence
\begin{equation}
	\Tree_{[b,a]}
	=\coprod_{r=a+1}^{b}
	\Tree_{[b,r]}\times\Tree_{[r-1,a]}.
	\label{eq:tree-recursion}
\end{equation}

For $T\in\Tree_{[b,a]}$, recursively apply
\eqref{eq:block-decomposition} at each internal vertex of $T$.  At a
root cut $r$, apply the upper subtree to $U_r^{[b,a]}x$ and the lower
subtree to $L_{r-1}^{[b,a]}x$.  The leaf labelled $j$ produces a
functor
\[
P_{T,j}:\mD_{[b,a]}\longrightarrow\Sigma^j\mA.
\]
Define the layer functor obtained along the path $T$ by
\[
H_{T,j}:=\Sigma^{-j}P_{T,j}:
\mD_{[b,a]}\longrightarrow\mA.
\]

\begin{prop}
	\label{prop:catalan-path-independence}
	Let $a\leq b$ with $b-a\leq\ell$.
	\begin{enumerate}
		\item The number of interval decomposition trees for $[b,a]$ is
		\[
		\lvert\Tree_{[b,a]}\rvert
		=C_{b-a}
		=\frac{1}{b-a+1}\binom{2(b-a)}{b-a}.
		\]
		In particular, the window $\mD_\ell=\mD_{[\ell,0]}$ has
		$C_\ell$ recursive binary extraction schemes.
		
		\item For every $T\in\Tree_{[b,a]}$ and every $a\leq j\leq b$,
		there is a canonical natural isomorphism
		\[
		\eta_{T,j}:H_{T,j}
		\xrightarrow{\ \sim\ }H_j^{[b,a]}.
		\]
		Consequently, every recursive binary extraction scheme produces the
		same layer homology functors, up to canonical natural isomorphism.
		
		\item These comparisons are coherent.  If $S,T\in\Tree_{[b,a]}$,
		put
		\[
		\eta_{S,T,j}:=\eta_{T,j}^{-1}\eta_{S,j}:
		H_{S,j}\xrightarrow{\ \sim\ }H_{T,j}.
		\]
		Then
		\[
		\eta_{T,T,j}=1,
		\qquad
		\eta_{T,R,j}\eta_{S,T,j}=\eta_{S,R,j}
		\]
		for all $R,S,T$.  Thus the comparison is independent of a chosen
		sequence of associativity rotations between the two trees.
	\end{enumerate}
\end{prop}

\begin{proof}
	Write $N_n$ for the number of trees on an interval containing $n$
	layers.  The recursive definition gives $N_1=1$ and
	\[
	N_n=\sum_{p=1}^{n-1}N_pN_{n-p}
	\qquad(n\geq2).
	\]
	This is the Catalan recurrence, so $N_n=C_{n-1}$.  Taking
	$n=b-a+1$ proves (1).
	
	We prove (2) by induction on $b-a+1$.  The assertion is the identity
	when $a=b$.  Suppose $a<b$, and let the root of $T$ be the cut $r$,
	with upper and lower subtrees $T^+$ and $T^-$.  If $r\leq j\leq b$,
	then the recursive definition, the induction hypothesis, and
	Lemma~\ref{lem:block-restriction} give
	\begin{align*}
		H_{T,j}
		&=H_{T^+,j}U_r^{[b,a]}\\
		&\xrightarrow{\sim}
		H_j^{[b,r]}U_r^{[b,a]}
		\xrightarrow{\sim}H_j^{[b,a]}.
	\end{align*}
	If $a\leq j<r$, the same argument on the lower subtree gives
	\begin{align*}
		H_{T,j}
		&=H_{T^-,j}L_{r-1}^{[b,a]}\\
		&\xrightarrow{\sim}
		H_j^{[r-1,a]}L_{r-1}^{[b,a]}
		\xrightarrow{\sim}H_j^{[b,a]}.
	\end{align*}
	These are the maps $\eta_{T,j}$.  They are natural, and the nested
	compatibility in Lemma~\ref{lem:block-restriction} shows that they do
	not depend on how the induction is written.
	
	Finally, the formulas in (3) immediately give the identity and
	cocycle relations.  Equivalently, an elementary rotation
	\[
	(X*Y)*Z\longleftrightarrow X*(Y*Z)
	\]
	changes the recursive decomposition tree but not its comparison with
	the canonical layer functors.  Since all full binary trees are
	connected by such rotations, every rotation path gives the same
	comparison.
\end{proof}

\section{Exactness of homology}\label{section 4}

The global hypothesis remains exactly \eqref{eq:global-vanishing}. If
$X\in\mD_{[r,0]}$, write its top-layer decomposition triangle as
\begin{equation}\label{eq:top-truncation}
	\Sigma^rH_r^{[r]}(X)\xrightarrow{\iota_X^r}X
	\xrightarrow{\pi_X^r}X^{<r}
	\longrightarrow\Sigma^{r+1}H_r^{[r]}(X),
	\qquad X^{<r}\in\mD_{[r-1,0]}.
\end{equation}
The superscript is suppressed whenever the ambient $\mD_r$ is clear.

\begin{thm}[Natural long exact sequence]\label{thm:main-long-exact}
	Let
	\begin{equation}\label{eq:main-triangle}
		\Delta=(X\xrightarrow fY\xrightarrow gZ\xrightarrow h\Sigma X)
	\end{equation}
	be a triangle with $X,Y,Z\in\mD_\ell$.  The construction in equation
	(\ref{eq:cross-window-boundary}) 
	gives canonical morphisms
	\[
	\delta_i^\Delta:H_i(Z)\longrightarrow H_{i-1}(X)
	\qquad(1\leq i\leq\ell).
	\]
	They make the sequence
	\begin{align*}
		0\longrightarrow H_\ell(X)&\longrightarrow H_\ell(Y)
		\longrightarrow H_\ell(Z)\xrightarrow{\delta_\ell^\Delta}
		H_{\ell-1}(X)\longrightarrow\cdots\\
		&\cdots\longrightarrow H_0(X)\longrightarrow H_0(Y)
		\longrightarrow H_0(Z)\longrightarrow0
	\end{align*}
	exact in $\mA$.  Moreover, the sequence is natural with respect to
	morphisms of triangles.
\end{thm}

The connecting maps are constructed next from the same block
$s$-torsion decompositions that define the $H_i$.  Exactness is then
proved in three steps, and naturality is proved separately. 

\begin{remark}
	For $\ell=1$, the hypothesis \eqref{eq:global-vanishing} is precisely $(E_2)$,
	and Theorem~\ref{thm:main-long-exact} recovers Korteg{\aa}rd's six-term exact sequence
	\cite[Lemma~3.3]{Kortegaard2025}.
\end{remark}

\subsection{Layer maps between adjacent shifted windows}

Fix $1\leq r\leq\ell$ and put
\[
\mD_r^+:=\Sigma\mD_r=\mD_{[r+1,1]}.
\]
Shifting the block decompositions of Section~\ref{section 3} gives, for
$Y\in\mD_r^+$,
\begin{align*}
	U_i^{+,[r]}(Y)&:=\Sigma U_{i-1}^{[r]}(\Sigma^{-1}Y),\\
	L_{i-1}^{+,[r]}(Y)&:=\Sigma L_{i-2}^{[r]}(\Sigma^{-1}Y),\\
	H_i^{+,[r]}(Y)&:=H_{i-1}^{[r]}(\Sigma^{-1}Y),
\end{align*}
with the evident endpoint conventions.  In particular, there are natural
shift isomorphisms
\begin{equation}\label{eq:shifted-layer-iso}
	\sigma_{i,x}:H_i^{+,[r]}(\Sigma x)\xrightarrow{\sim} H_{i-1}^{[r]}(x)
	\qquad(1\leq i\leq r).
\end{equation}

\begin{lem}\label{lem:cross-window}
	Every morphism
	\[
	u:x\longrightarrow y,
	\qquad x\in\mD_r,\quad y\in\mD_r^+,
	\]
	induces, for $1\leq i\leq r$, a unique morphism
	\[
	H_i^\rightarrow(u):H_i^{[r]}(x)\longrightarrow H_i^{+,[r]}(y).
	\]
	The assignment is additive and natural in commutative squares from
	$\mD_r$ to $\mD_r^+$.  It is compatible with the zero-top-layer
	identifications of Proposition \ref{prop 3.7}.
\end{lem}

\begin{proof}
	The following three groups vanish:
	\begin{align}
		\mathrm{Hom}_{\mC}(\mD_{[r,i]},\mD_{[i-1,1]})&=0,                    \label{eq:cross-v1}\\
		\mathrm{Hom}_{\mC}(\mD_{[r,i]},\Sigma^{-1}\mD_{[i-1,1]})&=0,        \label{eq:cross-v2}\\
		\mathrm{Hom}_{\mC}(\Sigma\mD_{[r,i+1]},\Sigma^i\mA)&=0.             \label{eq:cross-v3}
	\end{align}
	After reduction to pure layers, the positive exponents are contained in
	$[1,r-1]$, $[2,r]$, and $[2,r]$, respectively; hence
	\eqref{eq:global-vanishing} implies all three vanishings.
	
	Let
	\[
	e_i:U_i^{[r]}(x)\longrightarrow x,
	\qquad e_i^+:U_i^{+,[r]}(y)\longrightarrow y
	\]
	be the block-decomposition maps.  Apply
	$\mathrm{Hom}(U_i^{[r]}(x),-)$ to the decomposition triangle of $y$ at the cut
	$i$.  Equations \eqref{eq:cross-v1} and \eqref{eq:cross-v2} give a
	unique lift
	\begin{equation}\label{eq:unique-upper-lift}
		u_i:U_i^{[r]}(x)\longrightarrow U_i^{+,[r]}(y),
		\qquad e_i^+u_i=ue_i.
	\end{equation}
	Use the endpoint map $u_{r+1}=0$ when $i=r$.  The uniqueness in
	\eqref{eq:unique-upper-lift} shows that $u_{i+1}$ and $u_i$ commute with
	the maps $d_i$ in the two upper layer towers.
	Axiom \textup{(TR3)} therefore gives a map between their pure quotients:
	\[
	\begin{tikzcd}[column sep=large]
		U_{i+1}^{[r]}(x) \arrow[r] \arrow[d,"u_{i+1}"'] &
		U_i^{[r]}(x) \arrow[r] \arrow[d,"u_i"'] &
		\Sigma^iH_i^{[r]}(x) \arrow[r] \arrow[d,dashed] &
		\Sigma U_{i+1}^{[r]}(x) \arrow[d,"\Sigma u_{i+1}"]\\
		U_{i+1}^{+,[r]}(y) \arrow[r] &
		U_i^{+,[r]}(y) \arrow[r] &
		\Sigma^iH_i^{+,[r]}(y) \arrow[r] &
		\Sigma U_{i+1}^{+,[r]}(y).
	\end{tikzcd}
	\]
	The dashed map is unique by \eqref{eq:cross-v3}.  After shifting by
	$-i$, define it to be $H_i^\rightarrow(u)$.
	
	For a commutative square $bu=u'a$, the two induced upper lifts satisfy
	the same defining equation, hence agree.  Uniqueness of the dashed map
	then proves naturality; the same argument proves additivity.  Stability
	under adding a zero top layer follows from the uniqueness statement in
	Proposition \ref{prop 3.7}.
\end{proof}

For a triangle
\begin{equation}\label{eq:cross-window-triangle}
	\Delta=(X\xrightarrow fY\xrightarrow gZ\xrightarrow h\Sigma X)
\end{equation}
in $\mD_r$, define
\begin{equation}\label{eq:cross-window-boundary}
	\delta_i^\Delta
	:=\sigma_{i,X}\circ H_i^\rightarrow(h)
	:H_i^{[r]}(Z)\longrightarrow H_{i-1}^{[r]}(X),
	\qquad1\leq i\leq r.
\end{equation}

\subsection{Exactness and naturality}
\mbox{}\par\nobreak\smallskip

\noindent\textbf{Step 1: Top-layer four-term exactness.}
\par\smallskip

\begin{lem}\label{lem:top-representability}
	For $0\leq r\leq\ell$, $X\in\mD_{[r,0]}$, and $a\in\mA$,
	composition with $\iota_X^r$ induces a natural isomorphism
	\[
	\HomA(a,H_r(X))\xrightarrow{\sim}\HomC(\Sigma^ra,X).
	\]
	Moreover,
	\[
	\HomC(\Sigma^{r+1}\mA,\mD_{[r,0]})=0.
	\]
\end{lem}

\begin{proof}
	The first assertion is the adjunction attached to
	$(\Sigma^r\mA,\mD_{[r-1,0]})$.  For the second, test the pure layers:
	\[
	\HomC(\Sigma^{r+1}\mA,\Sigma^j\mA)
	=\HomC(\Sigma^{r+1-j}\mA,\mA)=0
	\qquad(0\leq j\leq r),
	\]
	and use Lemma \ref{lem 3.2}.
\end{proof}

Let
\begin{equation}\label{eq:level-r-triangle}
	\Delta_r=(X\xrightarrow fY\xrightarrow gZ\xrightarrow h\Sigma X)
\end{equation}
be a triangle in $\mD_{[r,0]}$, where $1\leq r\leq\ell$.

\begin{lem}\label{lem:edge-characterization}
	With the standard Verdier rotation convention, the top boundary from
	\eqref{eq:cross-window-boundary} is characterized by
	\begin{equation}\label{eq:edge-characterization}
		\iota_{X^{<r}}^{r-1}\Sigma^{r-1}\delta_r^{\Delta_r}
		=\pi_X^r\,\Sigma^{-1}(h\iota_Z^r).
	\end{equation}
	Here $H_{r-1}(X^{<r})\cong H_{r-1}(X)$ is the natural comparison from
	Section~\ref{section 3}.  In particular,
	\eqref{eq:edge-characterization} uniquely determines
	$\delta_r^{\Delta_r}$.
\end{lem}

\begin{proof}
	At the cut $r$, the lift \eqref{eq:unique-upper-lift} is a map
	\[
	u_r:\Sigma^rH_r(Z)=U_r^{[r]}(Z)\longrightarrow
	\Sigma U_{r-1}^{[r]}(X)
	\]
	such that $(\Sigma e_{r-1}^X)u_r=h\iota_Z^r$.  Put
	$\bar u_r:=\Sigma^{-1}u_r$.  The defining morphism of upper layer
	triangles gives
	\[
	\Sigma^{r-1}\delta_r^{\Delta_r}=\beta_{r-1}^X\bar u_r.
	\]
	Apply the nested-block square \eqref{eq:nested-block-square} with
	$i=r-1$; there $p_{r-1}^X=\pi_X^r$ and
	$\alpha_{r-1}^X=\iota_{X^{<r}}^{r-1}$.  We obtain
	\begin{align*}
		\iota_{X^{<r}}^{r-1}\Sigma^{r-1}\delta_r^{\Delta_r}
		&=\alpha_{r-1}^X\beta_{r-1}^X\bar u_r\\
		&=\pi_X^r e_{r-1}^X\bar u_r
		=\pi_X^r\Sigma^{-1}(h\iota_Z^r),
	\end{align*}
	which proves the identity.  Lemma~\ref{lem:top-representability} applied to
	$X^{<r}\in\mD_{r-1}$ proves uniqueness.  A different global rotation
	convention changes all connecting maps by the corresponding standard
	sign and does not affect exactness or naturality.
\end{proof}

\begin{lem}\label{lem:edge-monomorphism}
	If $X,Y\in\mD_{[r-1,0]}$ in \eqref{eq:level-r-triangle}, then
	$\delta_r^{\Delta_r}$ is monic.
\end{lem}

\begin{proof}
	Let $u:a\to H_r(Z)$ satisfy $\delta_r^{\Delta_r}u=0$.  Since the
	top-layer decomposition of $X$ is trivial, \eqref{eq:edge-characterization}
	gives $h\iota_Z^r\Sigma^ru=0$.  Exactness of
	$\HomC(\Sigma^ra,-)$ on the triangle supplies
	$s:\Sigma^ra\to Y$ with
	$gs=\iota_Z^r\Sigma^ru$.  But
	$\HomC(\Sigma^r\mA,\mD_{[r-1,0]})=0$, so $s=0$.
	Lemma~\ref{lem:top-representability} now gives $u=0$.  Thus
	$\HomA(a,\delta_r^{\Delta_r})$ is injective for every $a$, and
	$\delta_r^{\Delta_r}$ is monic.
\end{proof}

\begin{prop}
	\label{prop:top-four-term}
	For \eqref{eq:level-r-triangle}, the sequence
	\[
	0\longrightarrow H_r(X)\xrightarrow{H_r(f)}H_r(Y)
	\xrightarrow{H_r(g)}H_r(Z)
	\xrightarrow{\delta_r^{\Delta_r}}H_{r-1}(X)
	\]
	is exact at its first three nonzero terms.
\end{prop}

\begin{proof}
	Apply $\HomC(\Sigma^ra,-)$ to \eqref{eq:level-r-triangle}.  The preceding
	term $\HomC(\Sigma^{r+1}a,Z)$ vanishes by Lemma
	\ref{lem:top-representability}.  Lemma~\ref{lem:top-representability} followed by
	Yoneda detection in the abelian category $\mA$ proves exactness of
	\[
	0\longrightarrow H_r(X)\longrightarrow H_r(Y)\longrightarrow H_r(Z).
	\]
	
	Put $I=\mathrm{Im} H_r(g)$, let $j:I\hookrightarrow H_r(Z)$, and let
	$q_J:H_r(Z)\twoheadrightarrow J:=\mathrm{Coker} j$ be the quotient.  Properness realizes
	\[\begin{tikzcd}
		0 & {H_r(X)} & {H_r(Y)} & I & 0
		\arrow[from=1-1, to=1-2]
		\arrow[from=1-2, to=1-3]
		\arrow["\varepsilon", from=1-3, to=1-4]
		\arrow[from=1-4, to=1-5]
	\end{tikzcd}\]
	as a triangle.  The $3\times3$ lemma, applied to the natural top
	decomposition triangles of $X$ and $Y$, supplies an object $N$ and triangles
	\begin{align}
		X^{<r}&\longrightarrow Y^{<r}\longrightarrow N
		\xrightarrow{\nu}\Sigma X^{<r},                       \label{eq:relative-reduction}\\
		\Sigma^rI&\xrightarrow{t}Z\xrightarrow{\rho}N
		\longrightarrow\Sigma^{r+1}I.                         \label{eq:I-reduction}
	\end{align}
	
\[\begin{tikzcd}
	{\Sigma^rH_r(X)} & X & {X^{< r}} \\
	{\Sigma^rH_r(Y)} & Y & {Y^{< r}} \\
	{\Sigma ^rI} & Z & N \\
	& {\Sigma X} & {\Sigma X^{< r}}
	\arrow[from=1-1, to=1-2]
	\arrow[from=1-1, to=2-1]
	\arrow[from=1-2, to=1-3]
	\arrow[from=1-2, to=2-2]
	\arrow[from=1-3, to=2-3]
	\arrow[from=2-1, to=2-2]
	\arrow[from=2-1, to=3-1]
	\arrow[from=2-2, to=2-3]
	\arrow[from=2-2, to=3-2]
	\arrow[from=2-3, to=3-3]
	\arrow["{t}", from=3-1, to=3-2]
	\arrow["\rho", from=3-2, to=3-3]
	\arrow["h"', from=3-2, to=4-2]
	\arrow["\nu", from=3-3, to=4-3]
	\arrow["{\Sigma \pi_X^r}"', from=4-2, to=4-3]
\end{tikzcd}\]
	
	The following diagram gives $t=\iota_Z^r\Sigma^r j$.
	
	\[\begin{tikzcd}
		{\Sigma^rH_r(Y)} & {\Sigma^rI} & {\Sigma^{r+1}H_r(X)} & {} \\
		& Z
		\arrow["{\Sigma^r\varepsilon}", from=1-1, to=1-2]
		\arrow[from=1-2, to=1-3]
		\arrow["{t-\iota_Z^r\Sigma^rj}"', from=1-2, to=2-2]
		\arrow["{+1}", from=1-3, to=1-4]
		\arrow["0", dashed, from=1-3, to=2-2]
	\end{tikzcd}\]
	
	The octahedral axiom for
	$\Sigma^rI\to\Sigma^rH_r(Z)\to Z$ gives
	\begin{equation}\label{eq:N-top-truncation}
		\Sigma^rJ\xrightarrow{\iota_N^r}N\xrightarrow{\kappa}Z^{<r}
		\longrightarrow\Sigma^{r+1}J.
	\end{equation}
	
\[\begin{tikzcd}
	{\Sigma^rI} & {\Sigma^rH_r(Z)} & {\Sigma^r J} \\
	{\Sigma^rI} & Z & N \\
	& {Z^{< r}} & {Z^{< r}}
	\arrow[from=1-1, to=1-2]
	\arrow[equals, nfold, from=1-1, to=2-1]
	\arrow["{\Sigma^r q_J}", from=1-2, to=1-3]
	\arrow["{\iota_Z^r}"', from=1-2, to=2-2]
	\arrow["{\iota^r_N}", dashed, from=1-3, to=2-3]
	\arrow[from=2-1, to=2-2]
	\arrow["\rho", from=2-2, to=2-3]
	\arrow["{\pi_Z^r}", from=2-2, to=3-2]
	\arrow["\kappa", dashed, from=2-3, to=3-3]
	\arrow[equals, nfold, from=3-2, to=3-3]
\end{tikzcd}\]
	
	This is the top-layer decomposition triangle of $N$.  Hence
	\begin{equation}\label{eq:N-layers}
		H_r(N)\cong J,
		\qquad H_i(\kappa):H_i(N)\xrightarrow{\sim} H_i(Z^{<r})\quad(i<r).
	\end{equation}
	
	Choose the standard signs in the combined $3\times3$--octahedral diagram.
	Its relevant right-hand squares give
	\begin{equation}\label{eq:N-reduction-identities}
		\rho\iota_Z^r=\iota_N^r\Sigma^rq_J,
		\qquad
		\nu\rho=(\Sigma\pi_X^r)h,
		\qquad
		\kappa\rho=\pi_Z^r.
	\end{equation}
	Let $\Gamma$ denote the relative triangle
	\eqref{eq:relative-reduction}.  Applying Lemma~\ref{lem:edge-characterization} to
	$\Gamma$, precomposing by $\Sigma^{r-1}q_J$, and using
	\eqref{eq:N-reduction-identities}, we obtain
	\begin{align*}
		\iota_{X^{<r}}^{r-1}\Sigma^{r-1}(\delta_r^\Gamma q_J)
		&=\Sigma^{-1}(\nu\iota_N^r\Sigma^rq_J)\\
		&=\Sigma^{-1}(\nu\rho\iota_Z^r)\\
		&=\pi_X^r\Sigma^{-1}(h\iota_Z^r)
		=\iota_{X^{<r}}^{r-1}\Sigma^{r-1}\delta_r^{\Delta_r}.
	\end{align*}
	Applying Lemma~\ref{lem:top-representability} to $X^{<r}$ therefore gives
	$\delta_r^{\Delta_r}=\delta_r^\Gamma q_J$.  By Lemma
	\ref{lem:edge-monomorphism}, $\delta_r^\Gamma$ is monic.  Hence
	\[
	\mathrm{Ker}\delta_r^{\Delta_r}=I=\mathrm{Im} H_r(g).
	\]
\end{proof}

\par\medskip
\noindent\textbf{Step 2: A relative filtered snake lemma.}
\par\smallskip

The relative statement below and Step~3 are proved simultaneously by
induction on $r$.  When proving the relative statement at level $r$, we
use the general statement already established at level $r-1$.

\begin{prop}
	\label{prop:relative-snake}
	Let
	\[
	U\xrightarrow uV\xrightarrow vW\xrightarrow w\Sigma U
	\]
	be a triangle with $U,V\in\mD_{[r-1,0]}$ and
	$W\in\mD_{[r,0]}$, where $1\leq r\leq\ell$.  Assuming the long exact
	statement at level $r-1$, there is an exact sequence
	\begin{align*}
		0\longrightarrow H_r(W)&\xrightarrow{\delta_r}H_{r-1}(U)
		\longrightarrow H_{r-1}(V)\longrightarrow H_{r-1}(W)\\
		&\longrightarrow H_{r-2}(U)\longrightarrow\cdots
		\longrightarrow H_0(U)\longrightarrow H_0(V)
		\longrightarrow H_0(W)\longrightarrow0.
	\end{align*}
\end{prop}

\begin{proof}
	Set
	\[
	L=H_r(W),\qquad M=H_{r-1}(U),\qquad
	d=\delta_r:L\longrightarrow M.
	\]
	By Lemma \ref{lem:edge-monomorphism}, $d$ is monic.  Let
	$c:M\twoheadrightarrow K:=\mathrm{Coker} d$ be its cokernel.
	Properness gives a triangle
	\begin{equation}\label{eq:cokernel-edge}
		\Sigma^{r-1}L\xrightarrow{\Sigma^{r-1}d}\Sigma^{r-1}M
		\longrightarrow\Sigma^{r-1}K\longrightarrow\Sigma^rL.
	\end{equation}
	Let
	\[
	\alpha:=\Sigma^{-1}(w\iota_W^r):\Sigma^{r-1}L\longrightarrow U.
	\]
	By \eqref{eq:edge-characterization},
	$\alpha=\iota_U^{r-1}\Sigma^{r-1}d$.  The octahedral axiom applied to
	this factorization supplies $P$ and triangles
	\begin{align}
		\Sigma^{r-1}L&\xrightarrow{\alpha}U\xrightarrow{q}P
		\longrightarrow\Sigma^rL,                               \label{eq:P-first}\\
		\Sigma^{r-1}K&\longrightarrow P\longrightarrow U^{<r-1}
		\longrightarrow\Sigma^rK.                               \label{eq:P-top}
	\end{align}
	With the standard octahedral signs, the first octahedron also gives the
	morphism of top-layer triangles
	\begin{equation}\label{eq:P-top-comparison}
		\begin{tikzcd}[column sep=large]
			\Sigma^{r-1}M \arrow[r,"\iota_U^{r-1}"]
			\arrow[d,"\Sigma^{r-1}c"'] &
			U \arrow[r,"\pi_U^{r-1}"] \arrow[d,"q"'] &
			U^{<r-1} \arrow[r] \arrow[d,"1"'] &
			\Sigma^rM \arrow[d,"\Sigma^rc"]\\
			\Sigma^{r-1}K \arrow[r] &
			P \arrow[r] &
			U^{<r-1} \arrow[r] &
			\Sigma^rK.
		\end{tikzcd}
	\end{equation}
	The bottom row is the top-layer decomposition triangle of $P$, so
	\begin{equation}\label{eq:P-layers}
		H_{r-1}(P)\cong K,
		\qquad H_i(P)\cong H_i(U)\quad(i<r-1).
	\end{equation}
	Under these canonical identifications,
	\begin{equation}\label{eq:q-layer-maps}
		H_{r-1}(q)=c,
		\qquad H_i(q)\text{ is an isomorphism for }i<r-1.
	\end{equation}
	
	Apply the octahedral axiom to
	$\Sigma^rL\xrightarrow{\iota_W^r} W\xrightarrow{w}
	\Sigma U$.  
	
	\[\begin{tikzcd}
		{\Sigma^r L} & W & {W^{<r}} \\
		{\Sigma^r L} & {\Sigma U} & {\Sigma P} \\
		& {\Sigma V} & {\Sigma V}
		\arrow["{\iota_W^r}", from=1-1, to=1-2]
		\arrow[equals, nfold, from=1-1, to=2-1]
		\arrow["{\pi_W^r}", from=1-2, to=1-3]
		\arrow["w", from=1-2, to=2-2]
		\arrow["\omega", from=1-3, to=2-3]
		\arrow["{\Sigma\alpha}"', from=2-1, to=2-2]
		\arrow["{\Sigma q}", from=2-2, to=2-3]
		\arrow["{\Sigma u}"', from=2-2, to=3-2]
		\arrow["{\Sigma s}", from=2-3, to=3-3]
		\arrow[equals, nfold, from=3-2, to=3-3]
	\end{tikzcd}\]
	
	  More precisely, it gives the morphism of triangles
	\begin{equation}\label{eq:relative-morphism-of-triangles}
		\begin{tikzcd}[column sep=large]
			U \arrow[r,"u"] \arrow[d,"q"'] &
			V \arrow[r,"v"] \arrow[d,"1"'] &
			W \arrow[r,"w"] \arrow[d,"\pi_W^r"'] &
			\Sigma U \arrow[d,"\Sigma q"]\\
			P \arrow[r,"s"'] &
			V \arrow[r,"\pi_W^rv"'] &
			W^{<r} \arrow[r,"\omega"'] &
			\Sigma P.
		\end{tikzcd}
	\end{equation}
	The induction hypothesis applied to the bottom triangle,
	together with \eqref{eq:P-layers}, gives an exact sequence beginning
	\[
	0\longrightarrow K\longrightarrow H_{r-1}(V)
	\longrightarrow H_{r-1}(W)\longrightarrow H_{r-2}(U)
	\longrightarrow\cdots.
	\]
	For $1\leq i<r$, the maps $H_i(\pi_W^r)$ and $H_{i-1}(q)$ are isomorphisms by
	\eqref{eq:q-layer-maps} and
	Proposition \ref{prop 3.7}.  Naturality of the cross-window
	maps in Lemma \ref{lem:cross-window} therefore gives
	\begin{equation}\label{eq:relative-boundary-conjugacy}
		H_{i-1}(q)\,\delta_i^{(U,V,W)}
		=\delta_i^{(P,V,W^{<r})}\,H_i(\pi_W^r).
	\end{equation}
	The first two squares of
	\eqref{eq:relative-morphism-of-triangles}, together with
	\eqref{eq:q-layer-maps}, identify all ordinary layer maps in the inductive
	sequence with $H_i(u)$ and $H_i(v)$.  Its last square and
	\eqref{eq:relative-boundary-conjugacy} identify its connecting maps with
	the canonical boundaries \eqref{eq:cross-window-boundary} of the original
	relative triangle.  (The boundary on the right is computed at level
	$r-1$ and then identified with level $r$ by the zero-upper-layer
	comparison.)  Splicing with
	$0\to L\xrightarrow{d}M\xrightarrow{c}K\to0$ now gives exactly the asserted
	sequence.
\end{proof}

\par\medskip
\noindent\textbf{Step 3: The general triangle.}
\par\smallskip

\begin{prop}
	\label{prop:existence-long-exact}
	For every $0\leq r\leq\ell$ and every triangle
	\[
	\Delta=(X\xrightarrow fY\xrightarrow gZ\xrightarrow h\Sigma X)
	\]
	in $\mD_{[r,0]}$, the canonical connecting maps
	\eqref{eq:cross-window-boundary} give an exact sequence
	\begin{align*}
		0\longrightarrow H_r(X)&\longrightarrow H_r(Y)\longrightarrow H_r(Z)
		\longrightarrow H_{r-1}(X)\longrightarrow\cdots\\
		&\cdots\longrightarrow H_0(X)\longrightarrow H_0(Y)
		\longrightarrow H_0(Z)\longrightarrow0.
	\end{align*}
	In particular, for each $i$ the sequence
	\[
	H_i(X)\longrightarrow H_i(Y)\longrightarrow H_i(Z)
	\]
	is exact at $H_i(Y)$.
\end{prop}

\begin{proof}
	We induct on $r$.  At $r=0$, the assertion is precisely properness of
	$\mA$.  Assume it holds at $r-1$.  Apply
	Proposition \ref{prop:top-four-term} to the given triangle and retain the notation
	$I,J,N,q_J$ from its proof.  We obtain
	\[
	0\longrightarrow H_r(X)\longrightarrow H_r(Y)
	\longrightarrow H_r(Z)\longrightarrow J\longrightarrow0.
	\]
	The triangle \eqref{eq:relative-reduction} is relative in the sense of Proposition
	\ref{prop:relative-snake}.  By \eqref{eq:N-layers} and
	$\kappa\rho=\pi_Z^r$, the maps $H_i(\rho):H_i(Z)\to H_i(N)$ are
	isomorphisms for $0\leq i<r$.  Hence Proposition~\ref{prop:relative-snake} gives,
	after these canonical identifications,
	\[
	0\longrightarrow J\longrightarrow H_{r-1}(X)
	\longrightarrow H_{r-1}(Y)\longrightarrow H_{r-1}(Z)
	\longrightarrow\cdots\longrightarrow H_0(Z)\longrightarrow0.
	\]
	The $3\times3$ construction supplies the full morphism of triangles
	\begin{equation}\label{eq:general-reduction-morphism}
		\begin{tikzcd}[column sep=large]
			X \arrow[r,"f"] \arrow[d,"\pi_X^r"'] &
			Y \arrow[r,"g"] \arrow[d,"\pi_Y^r"'] &
			Z \arrow[r,"h"] \arrow[d,"\rho"'] &
			\Sigma X \arrow[d,"\Sigma\pi_X^r"]\\
			X^{<r} \arrow[r] &
			Y^{<r} \arrow[r] &
			N \arrow[r,"\nu"'] &
			\Sigma X^{<r}.
		\end{tikzcd}
	\end{equation}
	For $1\leq i<r$, the maps $H_{i-1}(\pi_X^r)$ are isomorphisms by Proposition
	\ref{prop 3.7}; moreover,
	$\kappa\rho=\pi_Z^r$ and \eqref{eq:N-layers} show that each
	$H_i(\rho)$ is an isomorphism.  Naturality in Lemma
	\ref{lem:cross-window} gives
	\begin{equation}\label{eq:general-boundary-conjugacy}
		H_{i-1}(\pi_X^r)\,\delta_i^{(X,Y,Z)}
		=\delta_i^{(X^{<r},Y^{<r},N)}\,H_i(\rho).
	\end{equation}
	The first two squares of \eqref{eq:general-reduction-morphism} show that,
	after the canonical lower-layer identifications, all ordinary maps in the
	relative sequence are $H_i(f)$ and $H_i(g)$.  The last square and
	\eqref{eq:general-boundary-conjugacy} give the same conclusion for all
	lower connecting maps.  At the splice, Step~1 already proved
	$\delta_r^\Delta=\delta_r^\Gamma q_J$.  Thus splicing the two displayed
	sequences gives exactly the canonical sequence claimed in the statement.
\end{proof}

\medskip
\noindent\textbf{Naturality of the connecting morphisms.}

\begin{prop}\label{prop:boundary-naturality}
	Let
	\[
	\begin{tikzcd}[column sep=large]
		X \arrow[r,"f"] \arrow[d,"x"'] &
		Y \arrow[r,"g"] \arrow[d,"y"'] &
		Z \arrow[r,"h"] \arrow[d,"z"'] &
		\Sigma X \arrow[d,"\Sigma x"] \\
		X' \arrow[r,"f'"'] &
		Y' \arrow[r,"g'"'] &
		Z' \arrow[r,"h'"'] &
		\Sigma X'
	\end{tikzcd}
	\]
	be a morphism of triangles in $\mD_r$, where $r\leq\ell$.  Then
	\[
	H_{i-1}^{[r]}(x)\,\delta_i^\Delta
	=\delta_i^{\Delta'}H_i^{[r]}(z)
	\qquad(1\leq i\leq r).
	\]
	Consequently the induced diagram of long exact sequences commutes.
\end{prop}

\begin{proof}
	The right-hand square of the morphism of triangles is
	$(\Sigma x)h=h'z$, a commutative square from $\mD_r$ to $\mD_r^+$.
	Apply naturality of $H_i^\rightarrow$ from Lemma
	\ref{lem:cross-window}, and then naturality of the shift isomorphism
	\eqref{eq:shifted-layer-iso}.  This gives
	\[
	H_{i-1}^{[r]}(x)\delta_i^\Delta
	=\delta_i^{\Delta'}H_i^{[r]}(z).
	\]
	The remaining squares commute by functoriality of the $H_i$.
\end{proof}

\begin{proof}[Completion of the proof of Theorem \ref{thm:main-long-exact}]
	Exactness is Step~3, and naturality is Proposition
	\ref{prop:boundary-naturality}.
\end{proof}

\begin{cor}\label{cor:arbitrary-window}
	Let $m<n$ be integers and put $d=n-m$. Assume that
	$\mathcal A$ satisfies $(E_{d+1})$, that is,
	\[
	\operatorname{Hom}_{\mathcal C}
	(\Sigma^p\mathcal A,\mathcal A)=0
	\qquad
	(1\leq p\leq d+1).
	\]
	Set
	\[
	\mathcal D_{[n,m]}
	:=
	\Sigma^n\mathcal A
	*\Sigma^{n-1}\mathcal A
	*\cdots*
	\Sigma^m\mathcal A.
	\]
	Then there are canonical additive functors
	\[
	H_i^{[n,m]}\colon
	\mathcal D_{[n,m]}\longrightarrow\mathcal A,
	\qquad
	m\leq i\leq n,
	\]
	and, for each triangle
	\[
	\Delta=
	\bigl(
	X\xrightarrow{f}Y\xrightarrow{g}Z
	\xrightarrow{h}\Sigma X
	\bigr)
	\]
	with $X,Y,Z\in\mathcal D_{[n,m]}$, canonical connecting
	morphisms
	\[
	\delta_i^\Delta\colon
	H_i^{[n,m]}(Z)\longrightarrow
	H_{i-1}^{[n,m]}(X),
	\qquad
	m<i\leq n,
	\]
	such that
	\[
	\begin{aligned}
		0
		&\longrightarrow H_n^{[n,m]}(X)
		\longrightarrow H_n^{[n,m]}(Y)
		\longrightarrow H_n^{[n,m]}(Z) \\
		&\xrightarrow{\delta_n^\Delta}
		H_{n-1}^{[n,m]}(X)
		\longrightarrow H_{n-1}^{[n,m]}(Y)
		\longrightarrow H_{n-1}^{[n,m]}(Z) \\
		&\longrightarrow\cdots
		\longrightarrow H_m^{[n,m]}(X)
		\longrightarrow H_m^{[n,m]}(Y)
		\longrightarrow H_m^{[n,m]}(Z)
		\longrightarrow 0
	\end{aligned}
	\]
	is exact in $\mathcal A$. Moreover, this sequence is natural
	with respect to morphisms of triangles.
\end{cor}

\begin{proof}
	Put
	\[
	\mathcal B:=\Sigma^m\mathcal A.
	\]
	Since $\Sigma^m$ is a triangle autoequivalence,
	$\mathcal B$ is again a full, replete, extension-closed proper
	abelian subcategory of $\mathcal C$. Moreover, for
	$1\leq p\leq d+1$, there are natural isomorphisms
	\[
	\begin{aligned}
		\operatorname{Hom}_{\mathcal C}
		(\Sigma^p\mathcal B,\mathcal B)
		&=
		\operatorname{Hom}_{\mathcal C}
		(\Sigma^{p+m}\mathcal A,\Sigma^m\mathcal A) \\
		&\cong
		\operatorname{Hom}_{\mathcal C}
		(\Sigma^p\mathcal A,\mathcal A)
		=0.
	\end{aligned}
	\]
	Thus $\mathcal B$ satisfies $(E_{d+1})$.
	
	The length-$d$ window generated by $\mathcal B$ is
	\[
	\begin{aligned}
		\mathcal D_d(\mathcal B)
		&=
		\Sigma^d\mathcal B
		*\Sigma^{d-1}\mathcal B
		*\cdots*
		\mathcal B \\
		&=
		\Sigma^n\mathcal A
		*\Sigma^{n-1}\mathcal A
		*\cdots*
		\Sigma^m\mathcal A
		=
		\mathcal D_{[n,m]}.
	\end{aligned}
	\]
	Apply Theorem~\ref{thm:main-long-exact} to the proper
	abelian subcategory $\mathcal B$ and the integer $d$.
	Let
	\[
	\widetilde H_j\colon
	\mathcal D_{[n,m]}\longrightarrow\mathcal B,
	\qquad
	0\leq j\leq d,
	\]
	be the resulting homology functors. For $m\leq i\leq n$,
	define
	\[
	H_i^{[n,m]}
	:=
	\Sigma^{-m}\widetilde H_{i-m}.
	\]
	The equivalence
	\[
	\Sigma^{-m}\colon\mathcal B\longrightarrow\mathcal A
	\]
	is exact with respect to the transported abelian structure.
	Applying it to the natural long exact sequence supplied by
	Theorem~\ref{thm:main-long-exact} gives precisely the
	displayed sequence. Naturality follows from the naturality in
	Theorem~\ref{thm:main-long-exact}.
\end{proof}

\begin{rem}
\label{rem:comparison-thesis}
	Korteg{\aa}rd's finite-window theorem
	\cite[Project D, Theorem~4.5]{Kortegard2024}
	constructs homology on the window
	\[
	\Sigma^{[n,m]}\mathcal A
	=
	\Sigma^n\mathcal A
	*\Sigma^{n-1}\mathcal A
	*\cdots*
	\Sigma^m\mathcal A
	\]
	and establishes a long exact sequence under the hypothesis $(E_{n-m+3})$.
	In contrast, Corollary~\ref{cor:arbitrary-window} requires only $(E_{n-m+1})$.
	Thus the present result removes the two highest vanishing
	requirements
	\[
	\operatorname{Hom}_{\mathcal C}
	\bigl(\Sigma^{n-m+2}\mathcal A,\mathcal A\bigr)=0
	\]
	and
	\[
	\operatorname{Hom}_{\mathcal C}
	\bigl(\Sigma^{n-m+3}\mathcal A,\mathcal A\bigr)=0.
	\]
	
	The difference also concerns the construction. The proof in
	\cite[Project D]{Kortegard2024} uses exact-category
	truncation methods, whereas the present argument uses only the
	homology functors obtained from the unique block
	$s$-torsion-pair decompositions and the adjacent-window maps
	of Section~\ref{section 4}.
\end{rem}

\subsection{Sharpness of the negative vanishing range}

\begin{prop}\label{prop:sharp}
	For every integer $\ell\geq1$, there is a Krull--Schmidt triangulated category
	$\mC$ with an extension-closed proper abelian subcategory $\mA$ satisfying
	$E_\ell$ but not \eqref{eq:global-vanishing} such that the following two requirements cannot
	hold simultaneously:
	\begin{enumerate}
		\item additive functors $H_i:D_\ell\to\mA$, $0\leq i\leq\ell$, satisfy
		$H_i(\Sigma^j a)\cong a$ for $i=j$ and $H_i(\Sigma^j a)=0$ otherwise;
		\item every triangle with its three terms in $D_\ell$ induces the finite long
		exact sequence in the form of Theorem \ref{thm:main-long-exact}.
	\end{enumerate}
\end{prop}

\begin{proof}
	Let $k$ be a field, put $N=\ell+1$, and take
	\[
	\mC=D^b(\mathrm{mod}\,k)/\Sigma^N.
	\]
	The orbit category is triangulated by Keller's theorem \cite{Keller2005}:
	$\mathrm{mod}\,k$ is hereditary, $\Sigma^N$ is standard, and its orbits meet
	a finite interval of shifts. Let $s$ be the image of $k$, and put
	$\mA=\add(s)$. The orbit formula gives
	\[
	\mathrm{Hom}_{\mC}(\Sigma^i s,\Sigma^j s)=
	\begin{cases}
		k,&i\equiv j\pmod N,\\
		0,&i\not\equiv j\pmod N.
	\end{cases}
	\]
	Every object is a finite direct sum of the $N$ distinct shifts of $s$,
	whose endomorphism rings are $k$. Thus $\mC$ is Krull--Schmidt.
	Since $N\geq2$, $	\mathrm{Hom}_{\mC}(\mA,\Sigma\mA)=0$, so $\mA$ is extension closed
	and proper abelian with its semisimple structure. The displayed formula
	proves $E_\ell$ and the failure of \eqref{eq:global-vanishing}.
	
	Rotate the identity triangle of $\Sigma^\ell s$. It gives a distinguished
	triangle
	\[
	\Sigma^\ell s\longrightarrow0\longrightarrow
	\Sigma^{\ell+1}s\xrightarrow{\sim}\Sigma^{\ell+1}s.
	\]
	The third term is isomorphic to $s$, and all three terms belong to $D_\ell$.
	Requirements (1) and (2) would give an exact beginning
	$0\to s\to0$, a contradiction.
\end{proof}

\begin{rem}
	This is sharpness of the \emph{uniform theorem with the stated normalization
		and endpoint exactness}. It does not assert that \eqref{eq:global-vanishing} is necessary
	for each particular pair $(\mC,\mA)$, or exclude other homological structures
	with different endpoint behavior.
\end{rem}

\section{Top-layer intermediate categories}

Throughout this section, we assume that $\mA$ is an extension-closed proper abelian subcategory in a Krull-Schmidt triangulated category $\mC$ satisfying \eqref{eq:global-vanishing}, where $\ell\geq 1$ is a positive integer. We write $\mA_i:=\Sigma^i\mA$ and, for $0\leq m\leq \ell$, set $\mathcal{D}_m:=\mA_m*\mA_{m-1}*\cdots* \mA_0, \mathcal{D}_0=\mA$.

\subsection{Intermediate subcategories and top-layer purity}

\begin{lem}\label{lem 5.2}\rm(Homotopy pullback)
	
	Let 
	\[\begin{tikzcd}
		b & x & c & {\Sigma b},
		\arrow["\alpha", from=1-1, to=1-2]
		\arrow[from=1-2, to=1-3]
		\arrow[from=1-3, to=1-4]
	\end{tikzcd}\]
	\[\begin{tikzcd}
		t & x & r & {\Sigma t}
		\arrow["\iota", from=1-1, to=1-2]
		\arrow["\psi", from=1-2, to=1-3]
		\arrow[from=1-3, to=1-4]
	\end{tikzcd}\]
	be distinguished triangles. Put $\beta:=\psi\alpha$, and complete $\beta$ to a distinguished triangle 
	\[\begin{tikzcd}
		m & b & r & {\Sigma m}.
		\arrow["\mu", from=1-1, to=1-2]
		\arrow["\beta", from=1-2, to=1-3]
		\arrow[from=1-3, to=1-4]
	\end{tikzcd}\]
	Then there are distinguished triangles 
	\[\begin{tikzcd}
		{\Sigma^{-1}c} & m & t & c,
		\arrow[from=1-1, to=1-2]
		\arrow[from=1-2, to=1-3]
		\arrow[from=1-3, to=1-4]
	\end{tikzcd}\]
	\[\begin{tikzcd}
		m & {b\oplus t} & x & {\Sigma m}
		\arrow["{(\mu, -v)}", from=1-1, to=1-2]
		\arrow["{(\alpha,\iota)}", from=1-2, to=1-3]
		\arrow[from=1-3, to=1-4]
	\end{tikzcd}\]
	for a morphism $v:m\to t$.
\end{lem}

\begin{proof}
	Apply the octahedral axiom to the composable maps $b\xrightarrow{\alpha}x\xrightarrow{\psi}r$. Rotating and shifting the third octahedral triangle gives the first required triangle.
	
	Consider the following morphism of triangles.
	\[\begin{tikzcd}
		m & b & r & {} \\
		t & x & r & {}
		\arrow["\mu", from=1-1, to=1-2]
		\arrow["v"', dashed, from=1-1, to=2-1]
		\arrow["\beta", from=1-2, to=1-3]
		\arrow["\alpha", from=1-2, to=2-2]
		\arrow["{+1}", from=1-3, to=1-4]
		\arrow[equals, from=1-3, to=2-3]
		\arrow["\iota"', from=2-1, to=2-2]
		\arrow["\psi"', from=2-2, to=2-3]
		\arrow["{+1}", from=2-3, to=2-4]
	\end{tikzcd}\]
	The homotopy-cartesian-square criterion says that the left square is homotopy cartesian, which gives the last required triangle.
\end{proof}

\begin{definition}
	Let $\mA$ be an extension-closed proper abelian subcategory in $\mC$. A subcategory $\mathcal{X}\subseteq \mC$ is called an $\ell$-layer intermediate subcategory if $\mD_{\ell-1}\subseteq \mX\subseteq \mD_{\ell}$ and $\mX$ is closed under extensions and direct summands. It is automatically closed under finite direct sums by the split triangle.
\end{definition}

\begin{lem}\label{lem 5.4}
	Let $\mX$ be an $\ell$-layer intermediate subcategory satisfying $\mX\subseteq \mD_{\ell-1}*\Sigma^\ell\mA$. Then $\Sigma^\ell H_\ell(x)\in \mX$, for every $x\in \mX$.
\end{lem}

\begin{proof}
	Fix $x\in \mX\subseteq \mD_{\ell-1}*\Sigma^\ell\mA$. By assumption, there is a distinguished triangle 
\[\begin{tikzcd}
	b & x & {\Sigma^\ell a} & {}
	\arrow["\alpha", from=1-1, to=1-2]
	\arrow[from=1-2, to=1-3]
	\arrow["{+1}", from=1-3, to=1-4]
\end{tikzcd}\]
where $b\in \mD_{\ell-1}, a\in \mA$.

The canonical highest-layer triangle is 
\[\begin{tikzcd}
	{\Sigma^\ell H_\ell(x)} & x & r & {}
	\arrow["\iota", from=1-1, to=1-2]
	\arrow["\psi", from=1-2, to=1-3]
	\arrow["{+1}", from=1-3, to=1-4]
\end{tikzcd}\]
where $r\in \mD_{\ell-1}$.

Complete $\beta:=\psi\alpha:b\to r$ to a triangle, applying the octahedral axiom we get the following diagram.
\[\begin{tikzcd}
	b & x & {\Sigma^\ell a} & {} \\
	b & r & {\Sigma m} \\
	& {\Sigma^{\ell+1}H_{\ell}(x)} & {\Sigma^{\ell+1}H_{\ell}(x)} \\
	& {} & {}
	\arrow["\alpha", from=1-1, to=1-2]
	\arrow[equals, from=1-1, to=2-1]
	\arrow[from=1-2, to=1-3]
	\arrow["\psi"', from=1-2, to=2-2]
	\arrow["{+1}", from=1-3, to=1-4]
	\arrow[dashed, from=1-3, to=2-3]
	\arrow["\beta"', from=2-1, to=2-2]
	\arrow[from=2-2, to=2-3]
	\arrow[from=2-2, to=3-2]
	\arrow[dashed, from=2-3, to=3-3]
	\arrow[equals, from=3-2, to=3-3]
	\arrow["{+1}", from=3-2, to=4-2]
	\arrow["{+1}", from=3-3, to=4-3]
\end{tikzcd}\]
Then $m\in \mA_{\ell-1}*\mA_\ell\subseteq \mA_\ell*\mA_{\ell-1}\subseteq \mD_{\ell}$ by Lemma \ref{lem 5.1}. Applying Theorem \ref{thm:main-long-exact} to the triangle $m\to b\xrightarrow{\beta} r\xrightarrow{+1}$, we have the following exact sequence in $\mA$.
\[\begin{tikzcd}
	0 & {H_\ell(m)} & {H_\ell(b)} & {H_\ell(r)}
	\arrow[from=1-1, to=1-2]
	\arrow[from=1-2, to=1-3]
	\arrow[from=1-3, to=1-4]
\end{tikzcd}\]
$H_\ell(b)=0=H_\ell(r)$ forces $H_\ell(m)=0$ since $b,r\in \mD_{\ell-1}$. Hence $m\in \mD_{\ell-1}\subseteq \mX$.

By Lemma \ref{lem 5.2}, there exists the following distinguished triangle.
\[\begin{tikzcd}
	m & {b\oplus\Sigma^\ell H_{\ell}(x)} & x & {}
	\arrow[from=1-1, to=1-2]
	\arrow[from=1-2, to=1-3]
	\arrow["{+1}", from=1-3, to=1-4]
\end{tikzcd}\]
$m, x\in \mX$ implies $b\oplus\Sigma^\ell H_\ell(x)\in \mX$, and its direct summand $\Sigma^\ell H_{\ell}(x)$ belongs to $\mX$ by definition.
\end{proof}

\subsection{Criteria for the highest-layer exchange equality}
The standing hypothesis $(E_{\ell+1})$ implies, by the extension closure
and associativity of extension, that
\begin{equation}
	\mD_{\ell-1}*\mA_\ell
	\subseteq
	\mA_\ell*\mD_{\ell-1}.
	\label{eq:automatic-exchange-inclusion}
\end{equation}
The purpose of this subsection is to give usable conditions under which the reverse inclusion also holds.

\begin{prop}
	\label{prop:top-exchange-factorization}
	Assume $(E_{\ell+1})$. The following conditions are equivalent.
	\begin{enumerate}[label=\textup{(\arabic*)}]
		\item
		The highest-layer exchange equality holds:
		\[
		\Sigma^\ell\mA*\mD_{\ell-1}
		=
		\mD_{\ell-1}*\Sigma^\ell\mA.
		\tag{$\mathrm{EX}_\ell$}
		\]
		
		\item
		Every morphism
		\[
		\gamma\colon d\longrightarrow\Sigma^{\ell+1}a,
		\qquad
		d\in\mD_{\ell-1},\quad a\in\mA,
		\]
		factors through an object of $\Sigma^\ell\mA$.
	\end{enumerate}
\end{prop}

\begin{proof}
	We first prove that \textup{(2)} implies \textup{(1)}. In view of
	\eqref{eq:automatic-exchange-inclusion}, it remains to establish
	\[
	\Sigma^\ell\mA*\mD_{\ell-1}
	\subseteq
	\mD_{\ell-1}*\Sigma^\ell\mA.
	\]
	Let $x\in\Sigma^\ell\mA*\mD_{\ell-1}$. Choose a distinguished
	triangle
	\begin{equation}
		\Sigma^\ell a\longrightarrow x\longrightarrow d
		\xrightarrow{\gamma}\Sigma^{\ell+1}a
		\label{eq:top-first-triangle}
	\end{equation}
	with $a\in\mA$ and $d\in\mD_{\ell-1}$. By
	\textup{(2)}, there are $b\in\mA$ and morphisms
	\[
	d\xrightarrow{u}\Sigma^\ell b
	\xrightarrow{v}\Sigma^{\ell+1}a
	\]
	such that $\gamma=vu$.
	
	Complete $u$ to a distinguished triangle
	\[
	m\longrightarrow d\xrightarrow{u}\Sigma^\ell b
	\longrightarrow\Sigma m.
	\]
	After rotation, this gives a distinguished triangle
	\[
	\Sigma^{\ell-1}b\longrightarrow m\longrightarrow d
	\longrightarrow\Sigma^\ell b.
	\]
	Both $\Sigma^{\ell-1}b$ and $d$ belong to $\mD_{\ell-1}$.
	Since $\mD_{\ell-1}$ is extension closed, it follows that
	\begin{equation}
		m\in\mD_{\ell-1}.
		\label{eq:m-in-lower-window}
	\end{equation}
	
	The shifted morphism $\Sigma^{-\ell}v\colon b\to\Sigma a$ can be
	completed to a distinguished triangle
	\[
	a\longrightarrow c\longrightarrow b
	\xrightarrow{\Sigma^{-\ell}v}\Sigma a.
	\]
	Because $\mA$ is extension closed, $c\in\mA$. Shifting this triangle
	by $\ell$ yields
	\[
	\Sigma^\ell a\longrightarrow\Sigma^\ell c
	\longrightarrow\Sigma^\ell b
	\xrightarrow{v}\Sigma^{\ell+1}a.
	\]
	
	Apply the octahedral axiom to the composable morphisms $u$ and $v$.
	Using \eqref{eq:top-first-triangle}, the two triangles just constructed
	occur as the fibre triangles of $u$, $v$, and $vu=\gamma$. The
	resulting octahedral triangle is
	\[
	m\longrightarrow x\longrightarrow\Sigma^\ell c
	\longrightarrow\Sigma m.
	\]
	By \eqref{eq:m-in-lower-window} and $c\in\mA$, this shows that
	\[
	x\in\mD_{\ell-1}*\Sigma^\ell\mA.
	\]
	The reverse inclusion is \eqref{eq:automatic-exchange-inclusion}, so
	the two extension products are equal.
	
	Conversely, assume \textup{(1)}. Let
	\[
	\gamma\colon d\longrightarrow\Sigma^{\ell+1}a,
	\qquad
	d\in\mD_{\ell-1},\quad a\in\mA,
	\]
	be any morphism, and complete it to a distinguished triangle
	\begin{equation}
		\Sigma^\ell a\xrightarrow{\alpha}x\xrightarrow{\beta}d
		\xrightarrow{\gamma}\Sigma^{\ell+1}a.
		\label{eq:triangle-from-obstruction}
	\end{equation}
	Then $x\in\Sigma^\ell\mA*\mD_{\ell-1}$. By \textup{(1)}, choose a
	distinguished triangle
	\begin{equation}
		d'\longrightarrow x\xrightarrow{q}\Sigma^\ell b
		\longrightarrow\Sigma d'
		\label{eq:reverse-decomposition}
	\end{equation}
	with $d'\in\mD_{\ell-1}$ and $b\in\mA$.
	
	There is a unique morphism $f\colon a\to b$ such that
	\[
	\Sigma^\ell f=q\alpha.
	\]
	Apply the highest-layer part of the long exact sequence to
	\eqref{eq:triangle-from-obstruction}. Since
	\[
	H_\ell(d)=0,
	\qquad
	H_\ell(\Sigma^\ell a)\cong a,
	\]
	the morphism
	\[
	H_\ell(\alpha)\colon a\xrightarrow{\sim}H_\ell(x)
	\]
	is an isomorphism. Applying the same exact sequence to
	\eqref{eq:reverse-decomposition}, and using $H_\ell(d')=0$, shows
	that
	\[
	H_\ell(q)\colon H_\ell(x)\longrightarrow b
	\]
	is a monomorphism. Under the pure-layer identifications,
	\[
	f=H_\ell(q)H_\ell(\alpha),
	\]
	so $f$ is a monomorphism in $\mA$.
	
	Let $c=\operatorname{Coker}f$. Properness of $\mA$ gives a
	distinguished triangle
	\[
	a\xrightarrow{f}b\longrightarrow c
	\xrightarrow{\eta}\Sigma a
	\]
	with $c\in\mA$. After shifting by $\ell$, we obtain
	\begin{equation}
		\Sigma^\ell a\xrightarrow{\Sigma^\ell f}\Sigma^\ell b
		\longrightarrow\Sigma^\ell c
		\xrightarrow{\Sigma^\ell\eta}\Sigma^{\ell+1}a.
		\label{eq:shifted-cokernel-triangle}
	\end{equation}
	Since $q\alpha=\Sigma^\ell f$, axiom \textup{(TR3)} completes the
	commutative square determined by $1_{\Sigma^\ell a}$ and $q$ to a
	morphism from \eqref{eq:triangle-from-obstruction} to
	\eqref{eq:shifted-cokernel-triangle}. In particular, there is a
	morphism
	\[
	u\colon d\longrightarrow\Sigma^\ell c
	\]
	such that
	\[
	\gamma=(\Sigma^\ell\eta)u.
	\]
	Thus $\gamma$ factors through an object of $\Sigma^\ell\mA$, proving
	\textup{(2)}.
\end{proof}

\begin{rem}
	When $\ell=1$, $(2)$ in Proposition \ref{prop:top-exchange-factorization} is \cite[Proposition 3.5(1)]{Kortegaard2025}.
\end{rem}

The factorization hypothesis is automatic under a stronger, but very easy to check, positive vanishing condition.

\begin{cor}
	[Positive vanishing criterion]
	\label{cor:positive-vanishing-exchange}
	Assume $(E_{\ell+1})$ and
	\[
	\mathrm{Hom}_{\mC}(\mA,\Sigma^q\mA)=0
	\qquad (2\leq q\leq\ell+1).
	\]
	Then
	\[
	\Sigma^\ell\mA*\mD_{\ell-1}
	=
	\mD_{\ell-1}*\Sigma^\ell\mA.
	\]
\end{cor}

\begin{proof}
	For $0\leq i\leq\ell-1$, one has
	\[
	\mathrm{Hom}_{\mC}
	\bigl(\Sigma^i\mA,\Sigma^{\ell+1}\mA\bigr)
	\cong
	\mathrm{Hom}_{\mC}
	\bigl(\mA,\Sigma^{\ell+1-i}\mA\bigr)
	=0,
	\]
	because $2\leq\ell+1-i\leq\ell+1$. Since the left orthogonal of
	$\Sigma^{\ell+1}\mA$ is closed under extensions, this gives
	\[
	\mathrm{Hom}_{\mC}
	\bigl(\mD_{\ell-1},\Sigma^{\ell+1}\mA\bigr)=0.
	\]
	Thus every morphism occurring in condition~\textup{(2)} of
	Proposition~\ref{prop:top-exchange-factorization} is zero, and the
	result follows.
\end{proof}

\section{Classification of top-layer intermediate subcategories}
Throughout this section, we assume that $\mA$ is an extension-closed proper abelian subcategory in a Krull-Schmidt triangulated category $\mC$ satisfying \eqref{eq:global-vanishing}, where $\ell\geq 1$ is a positive integer. We write $\mA_i:=\Sigma^i\mA$ and, for $0\leq m\leq \ell$, set $\mathcal{D}_m:=\mA_m*\mA_{m-1}*\cdots* \mA_0, \mathcal{D}_0=\mA$.

Recall that a torsion-free class in an abelian category is an additive subcategory closed under subobjects and extensions. It is automatically closed under direct summands.                  

\begin{lem}
	Let $\mF\subseteq\mA$ be a torsion-free class. For  $d\in \mD_{\ell-1}$, $a\in\mA$, $x\in \mD_\ell$, one has
		$$H_\ell(d)=0, \quad H_{\ell}(\Sigma^\ell a)\cong a,$$
		$$H_\ell(x)=0\iff x\in \mD_{\ell-1},$$
		and
		$$x\in \Sigma^\ell \mF*\mD_{\ell-1}\iff H_\ell(x)\in \mF.$$
\end{lem}

\begin{proof}
	This follows from the canonical top-layer
	triangle and its uniqueness.
\end{proof}

\begin{prop}\label{prop 6.1}
	If $\mF\subseteq \mA$ is a torsion-free class, then $$\mX_{\mF}:=\Sigma^\ell\mF*\mD_{\ell-1}$$
	is an $\ell$-layer intermediate subcategory. Moreover,
	$$H_\ell(\mX_{\mF})=\mF.$$
\end{prop}

\begin{proof}
	The inclusions $\mD_{\ell-1}\subseteq \mX_{\mF}\subseteq \mD_\ell$ are immediate.
	
	To prove extension closure, let 
	\[\begin{tikzcd}
		x & y & z & {}
		\arrow["f", from=1-1, to=1-2]
		\arrow["g", from=1-2, to=1-3]
		\arrow["{+1}", from=1-3, to=1-4]
	\end{tikzcd}\]
	be a distinguished triangle with $x, z\in \mX_{\mF}\subseteq \mD_\ell$. Proposition \ref{prop 3.3} implies $y\in \mD_\ell$. Applying Theorem \ref{thm:main-long-exact} to the above triangle, we have the following exact sequence in $\mA$.
	\[\begin{tikzcd}
		0 & {H_{\ell}(x)} & {H_{\ell}(y)} & {\mathrm{Im}H_{\ell}(g)} & 0
		\arrow[from=1-1, to=1-2]
		\arrow["{H_{\ell}(f)}", from=1-2, to=1-3]
		\arrow[from=1-3, to=1-4]
		\arrow[from=1-4, to=1-5]
	\end{tikzcd}\]
	One has $H_{\ell}(x), H_{\ell}(z)\in \mF$. $\mathrm{Im}H_{\ell}(g)$ is a subobject of $H_{\ell}(z)$ and $\mF$ is closed under subobjects, then $\mathrm{Im}H_{\ell}(g)\in \mF$. Since $\mF$ is extension-closed, then $H_{\ell}(y)\in \mF$, hence $y\in \mX_{\mF}$.
	
To prove closure under direct summands, note first that
\[
\operatorname{Hom}_{\mC}
\bigl(\Sigma^\ell \mF,\mD_{\ell-1}\bigr)
\subseteq
\operatorname{Hom}_{\mC}
\bigl(\Sigma^\ell \mA,\mD_{\ell-1}\bigr)
=0.
\]
Moreover, by induction, \(\mD_i\) is closed under direct summands for every
\(0\leq i\leq \ell\), by \cite[Proposition 2.1(i)]{Iyama2008}.
	
	By canonical distinguished triangle, we have $H_{\ell}(\mX_{\mF})\subseteq\mF$. Conversely, if $f\in\mF$, then $\Sigma^\ell f\in \mX_{\mF}$, using the split triangle with third term zero, and $H_{\ell}(\Sigma^\ell f)\cong f$.
\end{proof}

\begin{prop}\label{prop 6.2}
	Let $\mX$ be an $\ell$-layer intermediate subcategory satisfying $$\mX\subseteq\mD_{\ell-1}*\Sigma^\ell\mA.$$
	Define $$H_{\ell}(\mX):=\left\lbrace a\in \mA: a\cong H_{\ell}(x) \text{ for some } x\in \mX \right\rbrace. $$
	Then $H_{\ell}(\mX)$ is a torsion-free class in $\mA$, and $$\mX=\Sigma^\ell H_{\ell}(\mX)*\mD_{\ell-1}.$$
\end{prop}

\begin{proof}
	Set $\mF:=H_{\ell}(\mX)$. It contains zero since $0\in \mD_{\ell-1}\subseteq \mX$ and $H_{\ell}(0)=0$. If $f_i\cong H_{\ell}(x_i)$ with $x_i\in \mX$, then $x_1\oplus x_2\in \mX$ and additivity gives $$f_1\oplus f_2\cong H_{\ell}(x_1\oplus x_2)\in \mF.$$
	Thus $\mF$ is additive.
	
	To prove extension closure, let 
	\[\begin{tikzcd}
		0 & {f'} & e & {f''} & 0
		\arrow[from=1-1, to=1-2]
		\arrow[from=1-2, to=1-3]
		\arrow[from=1-3, to=1-4]
		\arrow[from=1-4, to=1-5]
	\end{tikzcd}\]
	be a short exact sequence in $\mA$ with $f', f''\in \mF$. Properness of $\mA$ turns the above into a distinguished triangle 
	\[\begin{tikzcd}
		{\Sigma^\ell f'} & {\Sigma^\ell e} & {\Sigma^\ell f''} & {}.
		\arrow[from=1-1, to=1-2]
		\arrow[from=1-2, to=1-3]
		\arrow["{+1}", from=1-3, to=1-4]
	\end{tikzcd}\]
	Then $\Sigma^\ell f', \Sigma^\ell f''\in \mX$ implies $\Sigma^\ell e\in\mX$ by Lemma \ref{lem 5.4}. Hence $e\cong H_{\ell}(\Sigma^\ell e)\in \mF$.
	
	To prove subobject closure, let u be a subobject of $f\in \mF$. In $\mA$ there is a short exact sequence 
	\[\begin{tikzcd}
		0 & u & f & q & 0.
		\arrow[from=1-1, to=1-2]
		\arrow[from=1-2, to=1-3]
		\arrow[from=1-3, to=1-4]
		\arrow[from=1-4, to=1-5]
	\end{tikzcd}\]
	The corresponding distinguished triangle, shifted by $\Sigma^\ell$ and rotated once to the right, is 
	\[\begin{tikzcd}
		{\Sigma^{\ell-1}q} & {\Sigma^\ell u} & {\Sigma^\ell f} & {}.
		\arrow[from=1-1, to=1-2]
		\arrow[from=1-2, to=1-3]
		\arrow["{+1}", from=1-3, to=1-4]
	\end{tikzcd}\]
	Here $\Sigma^{\ell-1}q\in \mD_{\ell-1}\subseteq\mX, \Sigma^\ell f\in \Sigma^\ell\mF=\Sigma^\ell H_{\ell}(\mX)\subseteq \mX$ by Lemma \ref{lem 5.4}. Then $\Sigma^\ell u\in \mX$, hence $u\cong H_{\ell}(\Sigma^\ell u)\in \mF$.
	
	$\mX\subseteq \Sigma^\ell H_{\ell}(\mX)*\mD_{\ell-1}$ is trivial by canonical distinguished triangle. Conversely, for any $y\in \Sigma^\ell H_{\ell}(\mX)*\mD_{\ell-1}$, suppose a triangle
	\[\begin{tikzcd}
		{\Sigma^{\ell}f} & y & d & {}
		\arrow[from=1-1, to=1-2]
		\arrow[from=1-2, to=1-3]
		\arrow["{+1}", from=1-3, to=1-4]
	\end{tikzcd}\]
	with $f\in H_{\ell}(\mX), d\in \mD_{\ell-1}$. The first term belongs to $\mX$ by Lemma \ref{lem 5.4}, and the third term belongs to $\mX$ because $\mD_{\ell-1}\subseteq\mX$. Extension closure gives $y\in\mX$.
\end{proof}

\begin{thm}\label{thm 6.3}
	Assume the highest-layer exchange equality
	$$\Sigma^\ell\mA*\mD_{\ell-1}=\mD_{\ell-1}*\Sigma^\ell\mA.$$
	Then there is a bijection
	\[
	\begin{array}{r@{\;}c@{\;}l}
		\left\{
		\begin{array}{c}
			\ell\text{-layer intermediate}\\
			\text{subcategories } \mathcal X
		\end{array}
		\right\}
		&
		\overset{1:1}{\longleftrightarrow}
		&
		\left\{
		\begin{array}{c}
			\text{torsion-free classes}\\
			\mathcal F\subseteq \mathcal A
		\end{array}
		\right\}
		\\[1.6em]
		\mathcal X
		&
		\mapsto
		&
		H_{\ell}(\mathcal X)
		\\[0.6em]
		\Sigma^\ell\mathcal F * \mathcal D_{\ell-1}
		&
		\mapsfrom
		&
		\mathcal F
	\end{array}
	\]
	Moreover, this bijection is an isomorphism of posets with respect to inclusion.
\end{thm}

\begin{proof}
	The bijection can be obtained directly from Proposition \ref{prop 6.1} and Proposition \ref{prop 6.2}.  
	
	Both assignments preserve inclusion, since
	$\mathcal X\subseteq\mathcal Y$ implies
	$H_\ell(\mathcal X)\subseteq H_\ell(\mathcal Y)$, and
	$\mathcal F\subseteq\mathcal G$ implies
	$\Sigma^\ell\mathcal F*\mD_{\ell-1}
	\subseteq\Sigma^\ell\mathcal G*\mD_{\ell-1}$.
	Hence the bijection is an isomorphism of posets.
\end{proof}

\section{Compatible layers and intermediate subcategories}

\subsection{Why one torsion-free class is insufficient}

\medskip
\noindent\textbf{Why one torsion-free class is insufficient}

One may instead ask for a classification of all extension-closed, direct summand-closed subcategories satisfying $$\mA\subseteq\mX\subseteq\mD_\ell$$
by a single torsion-free class. This stronger statement is false even in a semisimple derived category.
\begin{Ctex}\label{Ctex 7.1}
	Let $k$ be a field, $\mA=\mathrm{Vect}_k$ the category of finite-dimensional vector spaces, and $\mathcal{T}=D^b(\mathrm{Vect}_k)$. Then
	$$\mD_2=\mathrm{add}(\mA\cup\Sigma\mA\cup\Sigma^2\mA).$$
	The following four distinct subcategories contain $\mA$:
	$$\mA,\quad \mathrm{add}(\mA*\Sigma\mA),\quad \mathrm{add}(\mA*\Sigma^2\mA),\quad \mD_2.$$
	They are closed under direct summands because a direct summand cannot acquire cohomology
	in a new degree. They are extension closed because, if the first and third terms of a distinguished
	triangle have zero cohomology outside a fixed subset of degrees, the long exact cohomology
	sequence forces the middle term to have the same vanishing.
	On the other hand, $\mathrm{Vect}_k$ has only two trivial torsion-free classes $0$ and $\mA$. Indeed, if a torsion-free
	class contains a nonzero vector space, subobject closure gives a one-dimensional vector space, and
	extension closure then gives every finite-dimensional vector space. Hence these four subcategories cannot be classified by the two choices of a single torsion-free class.
\end{Ctex}

In fact, by Theorem \ref{thm:main-long-exact}, functors $H_i:\mD_\ell\to \mA$ and connecting morphisms make 
\[\begin{tikzcd}[column sep=small]
	0 & {H_{\ell}(c)} & {H_{\ell}(c')} & {H_{\ell}(c'')} & {H_{\ell-1}(c)} & \cdots & {H_{0}(c)} & {H_{0}(c')} & {H_{0}(c'')} & 0
	\arrow[from=1-1, to=1-2]
	\arrow["{H_{\ell}(f)}", from=1-2, to=1-3]
	\arrow["{H_{\ell}(g)}", from=1-3, to=1-4]
	\arrow["{\delta_{\ell}}", from=1-4, to=1-5]
	\arrow[from=1-5, to=1-6]
	\arrow[from=1-6, to=1-7]
	\arrow["{H_0(f)}", from=1-7, to=1-8]
	\arrow["{H_0(g)}", from=1-8, to=1-9]
	\arrow[from=1-9, to=1-10]
\end{tikzcd}\]
exact for every triangle $c\xrightarrow{f} c'\xrightarrow{g} c''\to\Sigma c$ in $\mD_\ell$.
For $1\leq i\leq \ell$, the relevant part is
\[\begin{tikzcd}
	{H_{i+1}(c'')} & {H_i(c)} & {H_i(c')} & {H_i(c'')} & {H_{i-1}(c)}
	\arrow["{\delta_{i+1}}", from=1-1, to=1-2]
	\arrow["{H_i(f)}", from=1-2, to=1-3]
	\arrow["{H_i(g)}", from=1-3, to=1-4]
	\arrow["{\delta_i}", from=1-4, to=1-5]
\end{tikzcd}\]
with $H_{\ell+1}=0$. Therefore 
\[\begin{tikzcd}
	0 & {\mathrm{Coker}\delta_{i+1}} & {H_i(c')} & {\mathrm{Ker}\delta_i} & 0
	\arrow[from=1-1, to=1-2]
	\arrow[from=1-2, to=1-3]
	\arrow[from=1-3, to=1-4]
	\arrow[from=1-4, to=1-5]
\end{tikzcd}\]
is exact in $\mA$.

At the highest layer, $\delta_{\ell+1}=0$, so the left term is $H_{\ell}(c)$. But at an interior layer, $\mathrm{Coker}\delta_{i+1}$ is a quotient of $H_i(c)$. A torsion-free class need not be closed under that quotient. This is exactly where the proof of Proposition \ref{prop 6.1} fails if one tries to impose
independent torsion-free conditions in every degree.

\subsection{Classification by compatible layer subcategories}\label{subxection 7.2}

Throughout this subsection, let $\mC$ be Krull--Schmidt and let $\mA$ be an
extension-closed proper abelian subcategory satisfying \eqref{eq:global-vanishing}.
Write
\[
D_\ell=\Sigma^\ell\mA*\cdots*\Sigma\mA*\mA,
\]
and use the functors $H_i:D_\ell\to\mA$.
Subcategories are full, replete, and additive. The finite windows are
closed under direct summands, as in  Proposition \ref{prop 6.1}.

\begin{definition}\label{compatible}
	An extension-closed, direct-summand-closed subcategory
	$\mA\subseteq\X\subseteq D_\ell$ is \emph{stable under taking pure homology
		layers} if
	\[
	x\in\X\quad\Longrightarrow\quad\Sigma^iH_i(x)\in\X
	\quad(0\leq i\leq\ell).
	\]
	A sequence $(\B_0,\ldots,\B_\ell)$ of additive, direct-summand-closed
	subcategories of $\mA$ is \emph{compatible} if:
	\begin{enumerate}
		\item $\B_0=\mA$ and each $\B_i$ is extension closed in $\mA$;
		\item for $1\leq i\leq\ell$ and every morphism
		$f:a\to b$ with $a\in\B_i$ and $b\in\B_{i-1}$, one has
		\[
		\mathrm{Ker} f\in\B_i,\qquad \mathrm{Coker} f\in\B_{i-1}.
		\tag{KC}
		\]
	\end{enumerate}
\end{definition}

\begin{thm}\label{thm:layers}
	There is an isomorphism of posets between compatible sequences, ordered
	componentwise by inclusion, and subcategories
	$\mA\subseteq\X\subseteq D_\ell$ closed under extensions, direct summands,
	and taking pure homology layers. The maps are
	\begin{align*}
		(\B_0,\ldots,\B_\ell)&\longmapsto
		\X_{\B}:=\{x\in D_\ell\mid H_i(x)\in\B_i
		\text{ for all }i\},\\
		\X&\longmapsto
		\bigl(\{a\in\mA\mid\Sigma^ia\in\X\}\bigr)_{i=0}^{\ell}.
	\end{align*}
	Moreover,
	\[
	\X_{\B}=\Sigma^\ell\B_\ell*\cdots*\Sigma\B_1*\B_0.
	\]
\end{thm}

\begin{proof}
	First let $\B$ be compatible. The category $\X_{\B}$ contains $\mA$ by
	normalization, and is additive and closed under direct summands by
	additivity of the $H_i$ and direct-summand closure of $D_\ell$.
	
	Consider a triangle $X\to Y\to Z\to\Sigma X$ with $X,Z\in\X_{\B}$.
	Extension closure of $D_\ell$ gives $Y\in D_\ell$. The long
	exact sequence gives, for $0\leq i\leq\ell$,
	\begin{equation}\label{eq 7.1}
		0\longrightarrow\mathrm{Coker}\delta_{i+1}
		\longrightarrow H_i(Y)\longrightarrow\mathrm{Ker}\delta_i
		\longrightarrow0,
	\end{equation}
	where $\delta_i:H_i(Z)\to H_{i-1}(X)$.
	At the endpoints use $\delta_{\ell+1}:0\to H_\ell(X)$ and
	$\delta_0:H_0(Z)\to0$.
	Condition (KC) places the cokernel and kernel in \eqref{eq 7.1} in $\B_i$;
	at the endpoints the corresponding assertion is immediate.
	Extension closure of $\B_i$ gives $H_i(Y)\in\B_i$, so
	$Y\in\X_{\B}$. Normalization shows that $\X_{\B}$ is stable under taking
	pure homology layers.
	
	Conversely, let $\X$ be as stated and set
	$\B_i=\{a\in\mA\mid\Sigma^ia\in\X\}$. Additivity and direct-summand closure
	are immediate, and $\B_0=\mA$. A short exact sequence in $\mA$ with its
	end terms in $\B_i$ becomes, after suspension, a triangle with end terms
	in $\X$. Properness and extension closure of $\X$ therefore show that
	$\B_i$ is extension closed.
	
	For $f:a\to b$ as in (KC), complete $\Sigma^if$ to a triangle
	\begin{equation}
		\Sigma^{i-1}b\longrightarrow y\longrightarrow
		\Sigma^i a\xrightarrow{\Sigma^if}\Sigma^i b.
	\end{equation}
	Its end terms lie in $\X$, so $y\in\X\subseteq D_\ell$. Normalization and
	the long exact sequence yield
	\begin{equation}
		0\longrightarrow H_i(y)\longrightarrow a
		\xrightarrow{\ \pm f\ } b\longrightarrow H_{i-1}(y)
		\longrightarrow0,
	\end{equation}
	and all other homology objects vanish. Here the boundary is $f$ up to
	the fixed rotation sign: the cross-window construction on pure objects
	is the ordinary shifted map. Hence $H_i(y)\cong\mathrm{Ker} f$ and
	$H_{i-1}(y)\cong\mathrm{Coker} f$. Stability under taking pure layers proves (KC).
	
	It remains to prove reconstruction. For $x\in\X$, stability gives
	$H_i(x)\in\B_i$, hence $x\in\X_{\B}$. If $x\in\X_{\B}$, all of its pure
	layers $\Sigma^iH_i(x)$ lie in $\X$. The canonical finite layer tower
	expresses $x$ as an iterated extension of these layers, so $x\in\X$.
	Normalization also recovers each $\B_i$ from $\X_{\B}$.
	The displayed extension-product formula follows from this tower and
	extension closure of $\X_{\B}$. Both assignments preserve inclusion.
\end{proof}

\begin{cor}
	\label{cor:top}
	Compatible sequences with $\B_i=\mA$ for $0\leq i<\ell$ are exactly those
	with $\B_\ell$ a torsion-free class in $\mA$. They classify the top-layer
	intermediate subcategories stable under taking their top pure layer.
	If the highest-layer exchange equality holds,  Lemma \ref{lem 5.4}
	shows that every top-layer intermediate subcategory has this stability.
\end{cor}

\begin{proof}
	Only (KC) at $i=\ell$ imposes a new restriction. If $\B_\ell$ is
	torsion free, kernels of maps from it are subobjects and belong to it;
	the cokernels belong to $\mA$. Conversely, for a subobject $u\subseteq a$
	with $a\in\B_\ell$, apply (KC) to $a\to a/u$ to obtain $u\in\B_\ell$.
	Extension closure is already part of compatibility. The final statement
	uses Lemma \ref{lem 5.4}; lower layers already lie in $\mD_{\ell-1}$.
\end{proof}

\begin{rem}
	The use of sequences of subcategories to control homology is established
	in the literature on narrow sequences and homology-determined preaisles,
	including Stanley--van Roosmalen \cite{Stanley2019}, Sakai \cite{Sakai2023} and Hanson~\cite{Hanson2024}.
	Theorem~\ref{thm:layers} is formulated for a finite window of a proper
	abelian subcategory, with neither an ambient $t$-structure nor shift
	closure of $\X$. Thus its sequences need not be nested and should not
	be identified with ICE sequences without additional arguments.
	The proof mechanism is closely related to the kernel--cokernel
	conditions behind those earlier classifications. 
\end{rem}

\subsection{Complete classification in split windows}

\begin{thm}
\label{thm:split}
Keep the assumptions of subsection \ref{subxection 7.2} and suppose in addition that
\[
\mathrm{Hom}_\mathcal{C}(\mA,\Sigma^q\mA)=0\qquad(2\leq q\leq\ell+1).
\tag{P}
\]
Then every $x\in D_\ell$ admits an isomorphism
\begin{equation}\label{composition}
	x\cong\bigoplus_{i=0}^{\ell}\Sigma^iH_i(x).
\end{equation}
The isomorphism need not be natural. Consequently,
Theorem~\ref{thm:layers} classifies \emph{all} extension-closed,
direct-summand-closed subcategories
$\mA\subseteq\X\subseteq D_\ell$.	
\end{thm}

\begin{proof}
	For $1\leq r\leq\ell$, the canonical top decomposition of an object
	$x\in D_r$ is
	\[
	\Sigma^rH_r(x)\longrightarrow x\longrightarrow x^{<r}
	\longrightarrow\Sigma^{r+1}H_r(x),\qquad x^{<r}\in D_{r-1}.
	\]
	Its connecting morphism vanishes: for $0\leq j\leq r-1$,
	\[
	\mathrm{Hom}_{\mC}(\Sigma^j\mA,\Sigma^{r+1}\mA)
	\cong\mathrm{Hom}_{\mC}(\mA,\Sigma^{r+1-j}\mA)=0
	\]
	by (P), and vanishing extends to $D_{r-1}$. Thus the triangle splits.
	Induction on $r$, using stability of lower homology along the residual
	tower, proves \eqref{composition}.
	Every direct-summand-closed $\X\subseteq D_\ell$ is now stable under
	taking pure layers. Theorem~\ref{thm:layers} applies.
\end{proof}

\begin{rem}
	For $\mA=\mathrm{mod}\,k$, every extension-closed additive subcategory of
	$\mA$ closed under direct summands is either $0$ or $\mA$. Every choice
	of $\B_1,\ldots,\B_\ell$ from these two categories satisfies (KC).
	Hence the classification gives $2^\ell$ subcategories containing $\mA$.
	For $\ell=2$, these are precisely the four subcategories in
	 Counterexample \ref{Ctex 7.1}. 
\end{rem}

\section{Examples and applications}

This section illustrates the homology construction and the classification results in three settings. For hearts of $t$-structures, our construction recovers
ordinary cohomology. We then consider negative cluster
categories and periodic derived orbit categories.
The latter give non-semisimple examples satisfying
\eqref{eq:global-vanishing} and $(P)$, with nonzero connecting morphisms.

\subsection{Hearts of \texorpdfstring{\(t\)}{t}-structures}

Let \((\mathcal C^{\leq0},\mathcal C^{\geq0})\) be a
\(t\)-structure on \(\mathcal C\), with heart
\(\mathcal{H}=\mathcal C^{\leq0}\cap\mathcal C^{\geq0}\), and write
\(\mathcal H^j\) for its cohomology functors.  The heart is a
proper abelian subcategory of \(\mathcal C\), and the orthogonality axiom of
the \(t\)-structure gives
\[
\operatorname{Hom}_{\mathcal C}(\Sigma^p\mathcal H,\mathcal H)=0
\qquad\text{for every }p\geq1.
\]
In particular, \eqref{eq:global-vanishing} holds for every \(\ell\geq1\).

\begin{prop}\label{prop:t-heart-identification}
	For \(0\leq i\leq\ell\), there is a natural isomorphism
	\[
	H_i|_{\mathcal D_\ell}
	\cong \mathcal H^{-i}|_{\mathcal D_\ell}.
	\]
	Under these isomorphisms, the connecting morphisms of Theorem \ref{thm:main-long-exact} agree with
	the usual connecting morphisms of cohomology.  Hence the long
	exact sequence of Theorem \ref{thm:main-long-exact} is the ordinary cohomology sequence restricted
	to \(\mathcal D_\ell\).
\end{prop}

\begin{proof}
	For \(x\in\mathcal D_n\), the standard truncation triangle at the cut \(i\)
	has one term in
	\(\mathcal D_{[n,i]}\) and the other in \(\mathcal D_{i-1}\).  It is
	therefore an \(s\)-torsion decomposition for the block pair
	\((\mathcal D_{[n,i]},\mathcal D_{i-1})\) of Proposition~\ref{prop 3.4}.  By uniqueness
	of \(s\)-torsion decompositions, it agrees naturally with the block triangle
	used in Section~\ref{section 3}.  Passing from two adjacent cuts identifies the pure layer
	in (\ref{eq:lower-layer-tower}) with
	\(\Sigma^i\mathcal H^{-i}(x)\), and hence identifies
	\(H_i(x)\) with \(\mathcal H^{-i}(x)\).
	
	Both connecting morphisms are characterized by the morphism from the shifted
	top layer of the third term to the adjacent layer of the first term induced
	by a morphism of truncation triangles.  The uniqueness used in Lemma~\ref{lem:cross-window}
	therefore identifies them.  Naturality follows either from this
	characterization or from Proposition~\ref{prop:boundary-naturality}.
\end{proof}

The classification theorem also specializes cleanly in the standard derived
setting.

\begin{cor}
	\label{cor:derived-heart-classification}
	Let \(\mathcal B\) be an abelian category for which
	\(\mathrm D^{\mathrm b}(\mathcal B)\) is Krull--Schmidt, let
	\(\mathcal C=\mathrm D^{\mathrm b}(\mathcal B)\), and identify
	\(\mathcal B\) with the standard heart.  For every \(\ell\geq1\),
	\[
	\Sigma^\ell\mathcal B*\mathcal D_{\ell-1}
	=\mathcal D_{\ell-1}*\Sigma^\ell\mathcal B.
	\]
	Consequently, Theorem~\ref{thm 6.3} gives an order-preserving bijection
	\[
	\left\{\text{torsion-free classes }\mathcal F\subseteq\mathcal B\right\}
	\overset{1:1}{\longleftrightarrow} 
	\left\lbrace \ell\text{-layer intermediate subcategories} \right\rbrace .
	\]
	Moreover, \(H_i(x)\cong \mathcal H^{-i}(x)\) for
	\(x\in\mathcal D_\ell\).
\end{cor}

\begin{proof}
	The standard Postnikov tower identifies \(\mathcal D_\ell\) with the objects
	whose cohomology is concentrated in degrees \([-\ell,0]\), and it gives
	\(\mathcal D_\ell=\Sigma^\ell\mathcal B*\mathcal D_{\ell-1}\).
	Choose a representative of such an object by a complex concentrated in
	degrees \([-\ell,0]\).  Brutal truncation at the left endpoint gives a
	triangle
	\[
	\sigma_{>-\ell}x\longrightarrow x\longrightarrow
	\Sigma^\ell x^{-\ell}\longrightarrow
	\Sigma\sigma_{>-\ell}x,
	\]
	where \(\sigma_{>-\ell}x\in\mathcal D_{\ell-1}\) and
	\(x^{-\ell}\in\mathcal B\).  Hence
	\(\mathcal D_\ell\subseteq
	\mathcal D_{\ell-1}*\Sigma^\ell\mathcal B\); the reverse inclusion follows
	from extension closure of \(\mathcal D_\ell\).  The result now follows from
	Theorem~\ref{thm 6.3} and Proposition~\ref{prop:t-heart-identification}.
\end{proof}

For an arbitrary heart in an arbitrary triangulated category, Proposition
\ref{prop:t-heart-identification} still holds, but the highest-layer exchange
equality need not be automatic.  Thus the homology theorem applies to every
heart, while the classification theorem applies when the additional exchange
condition of Theorem~\ref{thm 6.3} is available.

\subsection{Negative cluster categories}

Let \(Q\) be a finite acyclic quiver, let \(w\geq2\), and consider the
negative cluster category
\[
\mathcal C_{-w}(Q)
=\mathrm D^{\mathrm b}(\operatorname{mod}kQ)/
(\tau\Sigma^{w+1}).
\]
In the standard Hom-finite cases this is a Krull--Schmidt triangulated orbit
category \cite{Keller2005}.  Let \(\mathcal S\) be a \(w\)-simple-minded
system and put
\[
\mathcal A:=\langle\mathcal S\rangle,
\]
the extension closure of \(\mathcal S\).  J\o rgensen's theorem shows that
\(\mathcal A\) is a finite-length proper abelian subcategory, usually not the
heart of a \(t\)-structure \cite{Jorgensen2022}.  The defining
\(w\)-orthogonality of \(\mathcal S\), together with extension closure in
both variables, gives
\[
\operatorname{Hom}_{\mathcal C_{-w}(Q)}
(\Sigma^p\mathcal A,\mathcal A)=0
\qquad (1\leq p\leq w-1).
\]

\begin{prop}
	\label{prop:negative-cluster-homology}
	Let \(1\leq\ell\leq w-2\).  For
	\[
	\mathcal D_\ell
	=\Sigma^\ell\mathcal A*\cdots *\Sigma\mathcal A*\mathcal A,
	\]
	the functors \(H_0,\ldots,H_\ell\) and the natural long exact sequence of
	Theorem~\ref{thm:main-long-exact} are defined.  If, in addition,
	\[
	\Sigma^\ell\mathcal A*\mathcal D_{\ell-1}
	=\mathcal D_{\ell-1}*\Sigma^\ell\mathcal A,
	\]
	then Theorem~\ref{thm 6.3} identifies the poset of \(\ell\)-layer intermediate
	subcategories with the poset of torsion-free classes in \(\mathcal A\).
\end{prop}

\begin{proof}
	The inequality \(\ell\leq w-2\) is equivalent to
	\(\ell+1\leq w-1\), so the preceding orthogonality implies 
	\((E_{\ell+1})\).  Theorem~4.1 gives the first assertion.  The second follows
	from Theorem~6.4, since the negative cluster categories under consideration
	are Krull--Schmidt.
\end{proof}

The $w$-orthogonality of $\mathcal S$ gives
$(E_{\ell+1})$ for $\ell \leq w-2$.
The highest-layer exchange equality is an additional
hypothesis, verified in the following example.

\begin{prop}\label{prop:Cminus4-A1}
	Let $k$ be a field and let
	\[
	\mC=\mC_{-4}(A_1)
	=\mathrm D^{\mathrm b}(\modu k)/(\tau\Sigma^5).
	\]
	Let $s$ be the image of the one-dimensional $k$-module and set
	$\mA=\add(s)$.  Then the following statements hold.
	\begin{enumerate}[label=\textup{(\arabic*)},leftmargin=2.2em]
		\item There is a triangle equivalence
		\[
		\mC\simeq\mathrm D^{\mathrm b}(\modu k)/\Sigma^4,
		\]
		and
		\begin{equation}\label{eq:A1-Hom-formula}
			\mathrm{Hom}_{\mC}(\Sigma^is,\Sigma^js)
			\cong
			\begin{cases}
				k,&i\equiv j\pmod4,\\
				0,&i\not\equiv j\pmod4.
			\end{cases}
		\end{equation}
		\item The singleton $\mathcal S=\{s\}$ is a $4$-simple-minded system, and
		$\mA=\langle s\rangle$ is a semisimple proper abelian subcategory satisfying
		$(E_3)$.
		\item The subcategory $\mA$ is not the heart of any $t$-structure on $\mC$.
		\item For $\ell=2$ one has
		\[
		\mD_1=\add(\mA\cup\Sigma\mA),
		\qquad
		\mD_2=\add(\mA\cup\Sigma\mA\cup\Sigma^2\mA),
		\]
		and the highest-layer exchange equality holds:
		\[
		\Sigma^2\mA*\mD_1=\mD_1*\Sigma^2\mA=\mD_2.
		\]
		\item Every object of $\mD_2$ has a unique expression, up to isomorphism,
		\[
		x\cong V_0\oplus\Sigma V_1\oplus\Sigma^2V_2,
		\qquad V_i\in\modu k,
		\]
		and its layer homology is
		\[
		H_i(x)\cong V_i\qquad(0\leq i\leq2).
		\]
		Consequently, the two torsion-free classes $0$ and $\mA$ correspond under the
		top-layer classification to $\mD_1$ and $\mD_2$, respectively.
	\end{enumerate}
	
\end{prop}

\begin{proof}
	Since $\modu k$ is semisimple, the Serre functor of
	$\mathrm D^{\mathrm b}(\modu k)$ is the identity and hence
	$\tau=\Sigma^{-1}$.  Therefore $\tau\Sigma^5=\Sigma^4$, giving the first
	equivalence.  The orbit-category formula for morphisms gives
	\[
	\mathrm{Hom}_{\mC}(\Sigma^is,\Sigma^js)
	=\bigoplus_{q\in\mathbb Z}
	\mathrm{Hom}_{\mathrm D^{\mathrm b}(\modu k)}
	(\Sigma^is,\Sigma^{j+4q}s),
	\]
	which is $k$ precisely when $i\equiv j\pmod4$ and is zero otherwise.  This
	proves \eqref{eq:A1-Hom-formula}.
	
	We use the convention of \cite[Definition~1.2]{Jorgensen2022}.  In the
	singleton case, a set $\mathcal S=\{s\}$ is a $4$-simple-minded system if and
	only if
	\begin{align}
		&\mathrm{End}_{\mC}(s)\text{ is a division ring},
		\label{eq:SMS-division}\\
		&\mathrm{Hom}_{\mC}(s,\Sigma^rs)=0
		\qquad(r=-3,-2,-1),
		\label{eq:SMS-negative-orthogonality}\\
		&\mC
		=\langle s\rangle*\Sigma^{-1}\langle s\rangle
		*\Sigma^{-2}\langle s\rangle*\Sigma^{-3}\langle s\rangle,
		\label{eq:SMS-generation}
	\end{align}
	where $\langle s\rangle$ denotes the extension closure of $s$.
	
	\noindent\emph{Division-ring condition.}
	Taking $i=j=0$ in \eqref{eq:A1-Hom-formula} gives
	\[
	\mathrm{End}_{\mC}(s)\cong k.
	\]
	 Since $k$ is a field, this is a division ring, proving
	 \eqref{eq:SMS-division}.  There is no additional pairwise-orthogonality
	 condition to check because $\mathcal S$ has only one element.
	 
	 \smallskip
	 \noindent\emph{Negative orthogonality.}
	 If $r\in\{-3,-2,-1\}$, then $r\not\equiv0\pmod4$.  Formula
	 \eqref{eq:A1-Hom-formula} therefore gives
	 \[
	 \mathrm{Hom}_{\mC}(s,\Sigma^rs)=0.
	 \]
	 This is exactly \eqref{eq:SMS-negative-orthogonality}.  Equivalently, shifting
	 both variables by $-r$ gives
	 \[
	 \mathrm{Hom}_{\mC}(\Sigma^ps,s)=0
	 \qquad(1\leq p\leq3),
	 \]
	 which is the form used in the standing condition $(E_3)$.
	 
	 \smallskip
	 \noindent\emph{Generation.}
	 First observe that $\langle s\rangle=\add(s)$.  Indeed, a triangle defining an
	 extension of two objects in $\add(s)$ has connecting morphism in
	 \[
	 \mathrm{Hom}_{\mC}(\add(s),\Sigma\add(s))=0
	 \]
	 by \eqref{eq:A1-Hom-formula}; hence the triangle splits and its middle term is
	 again in $\add(s)$.  Thus $\add(s)$ is extension closed, while the reverse
	 inclusion $\add(s)\subseteq\langle s\rangle$ is immediate.
	 
	 Every bounded complex of finite-dimensional $k$-vector spaces is isomorphic in
	 $\mathrm D^{\mathrm b}(\modu k)$ to the direct sum of its cohomology objects,
	 placed in their respective degrees.  Hence every object of the orbit category
	 $\mC$ is a finite direct sum of shifts of $s$.  Since $\Sigma^4s\cong s$ in
	 $\mC$, the four objects
	 \[
	 s,\quad\Sigma^{-1}s,\quad\Sigma^{-2}s,\quad\Sigma^{-3}s
	 \]
	 represent all possible shifts.  Consequently, each $x\in\mC$ has an
	 isomorphism
	 \[
	 x\cong a_0\oplus\Sigma^{-1}a_1
	 \oplus\Sigma^{-2}a_2\oplus\Sigma^{-3}a_3,
	 \qquad a_i\in\add(s)=\langle s\rangle.
	 \]
	 Repeated split triangles place this direct sum in
	 \[
	 \langle s\rangle*\Sigma^{-1}\langle s\rangle
	 *\Sigma^{-2}\langle s\rangle*\Sigma^{-3}\langle s\rangle.
	 \]
	 This proves the inclusion from left to right in
	 \eqref{eq:SMS-generation}; the reverse inclusion is trivial because the
	 extension is formed inside $\mC$.  Hence \eqref{eq:SMS-generation} holds,
	 and $\mathcal S=\{s\}$ is a $4$-simple-minded system.
	 
	 Finally, $\mA=\langle s\rangle=\add(s)$ is naturally equivalent to
	 $\modu k$, so it is semisimple abelian.  Every short exact sequence in $\mA$
	 is split and therefore gives a split triangle in $\mC$.  Conversely, a
	 triangle whose three terms lie in $\mA$ is split because its connecting
	 morphism belongs to $\mathrm{Hom}_{\mC}(\mA,\Sigma\mA)=0$.  Thus $\mA$ is proper abelian, and the negative-orthogonality calculation above gives $(E_3)$.
	 
	 The category $\mA$ cannot be a heart.  Indeed, for the heart $\mathcal H$ of
	 any $t$-structure one has
	 \[
	 \mathrm{Hom}_{\mC}(\Sigma^p\mathcal H,\mathcal H)=0
	 \qquad(p\geq1).
	 \]
	 But $\Sigma^4s\cong s$ in $\mC$, so
	 \[
	 \mathrm{Hom}_{\mC}(\Sigma^4s,s)
	 \cong\mathrm{End}_{\mC}(s)\cong k\neq0.
	 \]
	 This is a contradiction.
	 
	 For the window calculation, consider a triangle defining
	 \(\Sigma\mA*\mA\). Its connecting morphism lies in
	 \[
	 \operatorname{Hom}_{\mC}
	 \bigl(\mA,\Sigma^2\mA\bigr),
	 \]
	 which vanishes by \eqref{eq:A1-Hom-formula}. Hence
	 \[
	 \mD_1=\Sigma\mA*\mA=\add(\mA\cup\Sigma\mA).
	 \]
	 Likewise, the connecting morphism of a triangle defining
	 $\Sigma^2\mA*\mD_1$ lies in
	 $\mathrm{Hom}_{\mC}(\mD_1,\Sigma^3\mA)=0$, so
	 \[
	 \mD_2=\Sigma^2\mA*\mD_1
	 =\add(\mA\cup\Sigma\mA\cup\Sigma^2\mA).
	 \]
	 Every object of the latter additive category belongs to
	 $\mD_1*\Sigma^2\mA$ by a split triangle.  Conversely,
	 $\mD_1*\Sigma^2\mA\subseteq\mD_2$ because $\mD_2$ is extension closed.
	 This proves the exchange equality.
	 
	 The Hom formula makes the three displayed layers pairwise orthogonal.  The
	 unique block $s$-torsion decompositions are therefore the obvious projections
	 onto their direct-sum factors, proving $H_i(x)\cong V_i$.  Finally,
	 $\mA\simeq\modu k$ has exactly two torsion-free classes, namely $0$ and
	 $\mA$.  The classification theorem sends them to
	 \[
	 \Sigma^2(0)*\mD_1=\mD_1
	 \quad\text{and}\quad
	 \Sigma^2\mA*\mD_1=\mD_2.
	 \]
\end{proof}

\begin{subsection}{Periodic derived orbit categories}

Let $k$ be a field, let $\mA$ initially denote the category of
finite-dimensional representations of the quiver $1\to 2$, and put
$N=\ell+3$ for $\ell\geq1$. Consider
\[
\mC_N=D^b(\mA)/\Sigma^N,
\]
and identify $\mA$ with its image in $\mC_N$.
Keller's theorem \cite{Keller2005} equips this orbit category with a
triangulated structure for which the projection is a triangle functor.

\begin{prop}
	\label{prop:periodic}
	The following assertions hold.
	\begin{enumerate}
		\item $\mC_N$ is Hom-finite and Krull--Schmidt, and $\mA$ is a
		non-semisimple extension-closed proper abelian subcategory.
		\item $\mA$ satisfies \eqref{eq:global-vanishing} and (P), but does not satisfy
		$E_{\ell+2}$.
		\item $\mA$ is not the heart of a $t$-structure on $\mC_N$.
		\item Theorem~\ref{thm:split} applies to $\mD_\ell$. In particular,
		all intermediate subcategories containing $\mA$ are classified by
		compatible sequences, and the top-layer interval has exactly five
		elements.
	\end{enumerate}
\end{prop}

\begin{proof}
	The category $\mA$ is hereditary. For $a,b\in\mA$ and $t\in\mathbb Z$,
	the orbit formula is
	\begin{equation}\label{eq 8.5}
		\mathrm{Hom}_{\mC_N}(a,\Sigma^t b)
		\cong\bigoplus_{q\in\mathbb Z}
		\mathrm{Hom}_{D^b(\mA)}(a,\Sigma^{t+qN}b).
	\end{equation}
	Only exponents $0$ and $1$ can contribute. In particular, for $t=0,1$
	the only possible contribution has $q=0$, yielding respectively
	$\mathrm{Hom}_{\mA}(a,b)$ and $\mathrm{Ext}^1_{\mA}(a,b)$.
	Every morphism $b\to\Sigma a$ is therefore the image of a short exact
	sequence in $\mA$. Its triangle has middle term in $\mA$, and any other
	triangle with the same boundary is isomorphic to it. This proves
	extension closure and properness.
	
	Every bounded complex over $\mA$ is a sum of shifted cohomology objects.
	There are three indecomposable representations, and shifts can be
	reduced modulo $N$. Formula \eqref{eq 8.5} shows that their shifted images have
	endomorphism ring $k$, and that all Hom spaces are finite-dimensional.
	This proves the first assertion, including the Krull--Schmidt property.
	
	For $1\leq p\leq N-2$, no integer $-p+qN$ equals $0$ or $1$.
	Consequently $E_{N-2}=E_{\ell+1}$ holds. For
	$2\leq t\leq N-2$, no integer $t+qN$ equals $0$ or $1$; this is (P).
	On the other hand,
	\begin{equation}\label{eq 8.6}
		\mathrm{Hom}_{\mC_N}(\Sigma^{N-1}a,b)\cong\mathrm{Ext}^1_{\mA}(a,b)
	\end{equation}
	and the latter is nonzero for suitable $a,b$. Thus
	$E_{N-1}=E_{\ell+2}$ fails.
	For any nonzero $a\in\mA$, periodicity also gives
	$\mathrm{Hom}_{\mC_N}(\Sigma^Na,a)\cong\mathrm{End}_{\mA}(a)\neq 0$.
	The negative orthogonality of a heart excludes $\mA$ from being a heart.

	For completeness, write
	$$
	S_1=(k\to0),\quad S_2=(0\to k),\quad P=(k\xrightarrow{1}k).
	$$
	These are the three indecomposables, and
	$0\to S_2\to P\to S_1\to0$ is nonsplit.
	The torsion-free classes are exactly
\begin{equation}\label{eq 8.7}
	0,\quad \add(S_1),\quad \add(S_2),\quad
	\add(S_2,P),\quad \mA.
\end{equation}
	Indeed, containing $P$ forces containing its subobject $S_2$, while
	containing both simples forces $P$ by extension closure.
	Conversely, the five displayed classes are subobject and extension
	closed. For $\add(S_2,P)$, one can also use its description as the
	representations whose arrow is injective; this property is preserved
	by subobjects and extensions. The other cases are immediate.
	
	Condition (P) and Theorem~\ref{thm:split} prove the complete
	compatible-layer classification. For the top-layer interval,
	Corollary~\ref{cor:top} and \eqref{eq 8.7} give its five elements
	$\Sigma^\ell\mathcal F*\mD_{\ell-1}$.
\end{proof}

\noindent A concrete nonzero connecting map is visible when $\ell=2$
and $N=5$. Rotate the nonsplit short exact sequence above and shift
once to obtain
\[
\Sigma S_1\longrightarrow\Sigma^2 S_2\longrightarrow
\Sigma^2 P\longrightarrow\Sigma^2 S_1.
\]
Its three terms lie in $\mD_2$, and its layer long exact sequence reduces
to
\[
0\longrightarrow S_2\longrightarrow P
\xrightarrow{\delta_2}S_1\longrightarrow0,
\]
where $\delta_2$ is the quotient map up to the fixed sign convention.
Thus the example exhibits both a non-semisimple abelian category and
a nonzero layer boundary. It satisfies $E_3$ but fails $E_4$, so it is
outside the $E_5$ hypothesis of the earlier finite-window theorem.

\begin{rem}
	The splitting in Theorem~\ref{thm:split} does not force the connecting
	morphisms in the long exact sequence to vanish, as the
	calculation above shows.
\end{rem}
\end{subsection}

\begingroup
\renewcommand{\addcontentsline}[3]{}

\section*{Declaration of generative AI and AI-assisted technologies}

During the preparation of this manuscript, the authors used
OpenAI's \mbox{GPT-5.6 Sol} and \mbox{GPT-6 Astra} to assist with language editing and improve the
exposition. The authors independently reviewed
the resulting suggestions and take full responsibility
for the mathematical content and the final manuscript.

\section*{Acknowledgements}

The authors thank Anders Sten Kortegård and Wei Hu
for helpful suggestions. The first author thanks Qinfeng Jiang for help with using OpenAI's Codex.

\vskip 20pt

\vskip 45pt

\bigskip

\noindent
\begin{minipage}{\textwidth}
	\normalfont
	\textbf{Zheng-Gang Wu}\\
	School of Mathematical Sciences, 
	Beijing Normal University\\
	Beijing 100875, China\\
	E-mail: 202631130011@mail.bnu.edu.cn
\end{minipage}

\bigskip

\noindent
\begin{minipage}{\textwidth}
	\normalfont
	\textbf{Xianhui Fu}\\
	School of Mathematics and Statistics,
	Northeast Normal University\\
	Changchun 130024, China\\
	E-mail: fuxianhui@gmail.com
\end{minipage}

\endgroup

\end{document}